\title{\bf The automorphism group of a graphon}
\author{\sc L\'aszl\'o Lov\'asz\footnote{Institute of Mathematics,
E\"otv\"os Lor\'and University, Budapest, Hungary. Research supported
by ERC Advanced Research Grant No.~227701.} \rm and \sc Bal\'azs
Szegedy\footnote{Alfr\'ed R\'enyi Mathematical Research Institute,
Budapest, Hungary. Research supported by ERC Consolidator Grant
No.~617747.}}
\date{ {\em Dedicated to the  memory of  \'Akos Seress}}

\documentclass{article}
\usepackage{amsmath,amssymb,graphicx,bbm,bm,delarray,pict2e,theorem}
\usepackage[hypertex]{hyperref}
\usepackage{srcltx}
\nonstopmode

\long\def\ignore#1{}
\def\proofend{\hfill$\square$}

\def\Aut{{\rm Aut}}

\def\hom{{\sf hom}}

\def\ul#1{[\hskip-0.12em[#1]\hskip-0.12em]}

\def\eps{\varepsilon}
\def\phi{\varphi}
\def\E{{\sf E}}

\def\rk{{\rm rk}}

\def\diag{{\rm diag}}

\def\AA{\mathcal{A}}\def\BB{\mathcal{B}}
\def\FF{\mathcal{F}}
\def\GG{\mathcal{G}}\def\II{\mathcal{I}}

\def\QQ{\mathcal{Q}}

\def\Gbb{\mathbb{G}}

\def\Nbb{\mathbb{N}}
\def\Rbb{\mathbb{R}}

\def\R{\Rbb}\def\N{\Nbb}
\def\one{{\mathbbm1}}

\newtheorem{theorem}{Theorem}
\newtheorem{prop}[theorem]{Proposition}
\newtheorem{lemma}[theorem]{Lemma}

\newtheorem{claim}{Claim}
\newtheorem{corollary}[theorem]{Corollary}

\theorembodyfont{\rmfamily}
\newtheorem{definition}[theorem]{Definition}
\newtheorem{example}[theorem]{Example}

\newtheorem{remark}[theorem]{Remark}

\newenvironment{proof}{\medskip\noindent{\bf Proof. }}{\hfill$\square$\medskip}
\newenvironment{proof*}[1]{\medskip\noindent{\bf Proof of #1.}}{\hfill$\square$\medskip}

\begin{document}

\maketitle

\tableofcontents

\medskip

\begin{abstract}
We study the automorphism group of graphons (graph limits). We prove
that after an appropriate ``standardization'' of the graphon, the
automorphism group is compact. Furthermore, we characterize the
orbits of the automorphism group on $k$-tuples of points. Among
applications we study the graph algebras defined by finite rank
graphons and the space of node-transitive graphons.
\end{abstract}

\section{Introduction}

Graphons have been introduced as limit objects of convergent
sequences of dense simple graphs, and many aspects of graphs can be
extended to graphons. The goal of this paper is to describe a natural
way to extend the notion of graph automorphisms to graphons. Our
notion of automorphism group satisfies the natural requirement that
it is invariant under weak isomorphism of graphons. (Weakly
isomorphic graphons represent the limit objects of the same
convergent graph sequences.) Thus our study of the automorphisms of
graphons fits well into graph limit theory.

In this paper we heavily use the topological aspects of graph limit
theory developed in \cite{LSz8}. It was shown in \cite{LSz8} that
every graphon has two ``canonical'' representations on metric spaces,
which we call, informally, the {\it neighborhood metric} and the {\it
$2$-neighborhood metric}. These metric spaces depend only on the weak
isomorphism class of the graphon. (In \cite{Hombook}, these are
called the ``neighborhood metric'' and the ``similarity metric''.)
The neighborhood metric space is simpler to define and work with, but
it is not compact in general; the $2$-neighborhood metric space is
compact. The automorphism group acts on each of these as a subgroup
of isometries. It is a rather straightforward consequence of the
compactness of the $2$-neighborhood metric that the automorphism
group is always a compact topological group (Theorem
\ref{THM:AUT-COMP}). This fact is also closely related to (and could
be derived from) a theorem of Vershik and Hab\"ock \cite{VH} on the
compactness of isometry groups of multivariate functions. As a
consequence we prove that for node-transitive graphons the
neighborhood metric is also compact.

The space of graphons (with weakly isomorphic graphons identified) is
compact in a natural topology (defined by the ``cut distance'').
Another result of this paper is that the set of node-transitive
graphons is closed, and hence compact, in this topology. As we will
see, graph limit theory restricted to this closed set gives rise to a
rather interesting limit theory for functions on groups. Such a
theory was initiated in \cite{Sz2}, and it was a crucial component of
the limit approach to higher order Fourier analysis (see \cite{Sz1}).

We give a characterization of the orbits of the automorphism group on
$k$-tuples of points. This generalizes results in \cite{L2} from
finite graphs to graphons, as well as the characterization of weak
isomorphism of graphons by Borgs, Chayes and Lov\'asz \cite{BCL}. We
use this characterization to connect the topic of graph algebras with
group theory. As an application, we give a group theoretic
description of the graph algebras defined by finite rank graphons.

It follows from our results that the limit of a convergent sequence
of finite graphs, each having a node-transitive automorphism group,
is a node-transitive graphon. However, the relationship between the
automorphism groups of the finite graphs and that of the limit
graphon is more involved.

\section{Preliminaries}

\subsection{Graphs and graphons}

A {\it $k$-labeled graph} is a graph (simple or multi) with $k$ of
its nodes labeled $1,\dots,k$ ($k\in\{0,1,2,\dots\}$). We denote by
$\FF_k$ the set of $k$-labeled multigraphs, by $\GG_k$ the set of
$k$-labeled simple graphs, and by $\GG_k^0$ the set of $k$-labeled
simple graphs with nonadjacent labeled nodes. In particular, $\GG_0$
is the set of unlabeled simple graphs.

We will need some special $k$-labeled graphs and multigraphs. We
denote by $K_2$ the (unlabeled) graph with two nodes and one edge,
and by $C_2$ the multigraph consisting of two nodes connected by two
edges. We denote by $K_2^{\bullet}$ and $K_2^{\bullet\bullet}$ the
graph $K_2$ with one and two nodes labeled, respectively;
$C_2^{\bullet}$ and $C_2^{\bullet\bullet}$ are defined analogously.
We denote by $P_n^{\bullet\bullet}$ the path with $n$ nodes, with its
two endpoints labeled.

For two simple graphs $F$ and $G$, let $\hom(F,G)$ denote the number
of homomorphisms (adjacency-preserving maps) $V(F)\to V(G)$. We
define the {\it homomorphism density}
\[
t(F,G)=\frac{\hom(F,G)}{|V(G)|^{|V(F)|}}.
\]

A {\it graphon} consists of a standard probability space $J$ and a
symmetric measurable function $W:~J\times J\to [0,1]$. To simplify
notation, we will omit some letters that may be understood. For the
standard probability space $J$, we let $\BB$ denote the underlying
sigma-algebra and let $\pi$ denote the probability measure. Also, we
write $dx$ instead of $d\pi(x)$ in integrals if there is only one
probability measure considered.

Every graphon $(J,W)$ defines an integral operator $T_W$ on the
Hilbert space $L^2(J)$ by
\[
(T_Wf)(x)=\int_J W(x,y)f(y)\,dy.
\]
We say that $W$ has finite rank if this operator has finite rank
(i.e., its range is a finite dimensional subspace of $L^2(J)$.

For every graphon $(J,W)$ and every graph $F=(V,E)$, we define
\begin{equation}\label{EQ:T-DEF}
t(F,J,W)=\int\limits_{J^V} \prod_{ij\in E} W(x_i,x_j)\,\prod_{i\in
V} dx_i.
\end{equation}
We note that the formula makes sense for multigraphs $F$, but we
exclude loops. We write $t(F,W)$ instead of $t(F,J,W)$ if the
underlying probability space is clear.

Graphons were introduced to describe limit objects of convergent
sequences of dense graphs. A sequence of simple graphs $G_n$ is
called {\it convergent}, if the numerical sequence $t(F,G_n)$
converges for every simple graph $F$. In this case, there is a
graphon $(J,W)$ such that $t(F,G_n)$ converges to $t(F,W)$ for every
simple graph $F$ \cite{LSz1}.

The limiting graphon is not strictly uniquely determined. Quite often
one assumes that $J=[0,1]$ (with the Lebesgue measure). In this
paper, different underlying spaces will be more useful. We say that
two graphons $(J,W)$ and $(J',W')$ are {\it weakly isomorphic}, if
$t(F,W)=t(F,W')$ for every simple graph $F$. Every graphon is weakly
isomorphic to a graphon on $[0,1]$, but this is not always the most
convenient representative of a weak isomorphism class.

Weakly isomorphic pairs of graphons were characterized in \cite{BCL}.
Let $J$ and $L$ be standard probability spaces and let $\phi:~J\to L$
be a measure preserving map. For any function $U:~L\times L\to\R$, we
define the function $U^\phi:~J\times J\to\R$ by
\[
U^\phi(x,y)=U(\phi(x),\phi(y)).
\]
It is clear that if $(L,U)$ is a graphon, then so is $(J,U^\phi)$,
which we call the {\it pullback} of $(L,U)$ along $\phi$. It is easy
to see that the graphons $(L,U)$ and $(J,U^\phi)$ are weakly
isomorphic. It follows that all pullbacks of the same graphon are
weakly isomorphic. The main result of \cite{BCL} asserts that two
graphons are weakly isomorphic if and only if they are pullbacks of
the same graphon.

We can define a sequence of graphons $(W_1,W_2,\dots)$ to be {\it
convergent} with limit graphon $W$ if $t(F,W_n)$ converges to
$t(F,W)$ for every simple graph $F$. (There is a semimetric, called
the ``cut distance'', on the set of graphons that makes this space
compact, and which defines this same notion of convergence. We don't
need the cut distance in this paper, however.)

We can define homomorphism densities of $k$-labeled graphs in
graphons, but these will be $k$-variable functions $J^k\to\R$ rather
than numbers. These {\it restricted homomorphism densities} are
defined by not integrating the variables corresponding to labeled
nodes:
\begin{equation}\label{EQ-REST-HOM}
t_{x_1,\dots,x_k}(F,W)=\int\limits_{J^{V\setminus[k]}}
\prod_{ij\in E} W(x_i,x_j)\,\prod_{i\in
V\setminus[k]} d\pi(x_i).
\end{equation}

\subsection{Graph algebras}

Graph algebras are important algebraic structures associated with
graph parameters. We give a quick introduction to the subject. For
more details see \cite{Hombook}.

For two simple graphs $G,H\in\mathcal{G}_k$, the product $GH$ is
defined as the graph obtained from $G$ and $H$ by identifying
vertices with the same label and by reducing multiple edges. This
product defines a commutative semigroup structure on $\mathcal{G}_k$.
Let $\mathcal{Q}_k$ denote the set of formal $\mathbb{R}$-linear
combinations of elements from $\mathcal{G}_k$. Such linear
combinations are usually called {\it quantum graphs}. The
multiplication extends to $\mathcal{Q}_k$ from $\mathcal{G}_k$ using
the distributive law and thus $\mathcal{Q}_k$ becomes a commutative
algebra. In other words, $\mathcal{Q}_k$ is the semigroup algebra of
$\mathcal{G}_k$.

Let $\ul{G}$ be the graph obtained by removing the labels in the
graph $G$. We extend this notation to quantum graphs by linearity. An
arbitrary graph parameter $f:~\mathcal{G}\rightarrow\mathbb{R}$ can
be extended to quantum graphs by linearity. Similarly, restricted
homomorphism densities can be extended to $k$-labeled quantum graphs
by linearity: if $f=\sum_{i=1}^n a_iF_i$, then
\[
t_{x_1,\dots,x_k}(f,W) = \sum_{i=1}^n a_it_{x_1,\dots,x_k}(F_i,W).
\]

Every graph parameter gives rise to a symmetric bilinear form on
$\QQ_k$ defined by $\langle G,H\rangle=f(\ul{GH})$. Let
$\mathcal{I}_k$ be the set of elements $Q$ in $\mathcal{Q}_k$ such
that $\langle Q,P\rangle_f=0$ for every $P\in\mathcal{Q}_k$. Then
$\mathcal{I}_k$ is an ideal and $\mathcal{Q}_k/\mathcal{I}_k$ is the
{\it graph algebra} corresponding to $f$. The infinite matrix
$M_k:~\mathcal{G}_k\times\mathcal{G}_k\rightarrow\mathbb{R}$ defined
by $M_k(G,H)=f(\ul{GH})=\langle G,H\rangle$ is called the {\it $k$-th
connection matrix} of $f$. It is easy to see that the rank of $M_k$
is the dimension of $\QQ_k/\II_k$.

We will be interested in the special case when $f$ is defined by
$f(G)=t(G,W)$ for some fixed graphon $W$. In this case, $\QQ_k/\II_k$
depends only on the weak isomorphism class of $W$. It was shown in
\cite{LSz1} that in this case the inner product $\langle.,.\rangle$
is positive semidefinite, and hence so are the connection matrices.

There is a concrete representation of $\QQ_k/\II_k$ that will be
convenient to use and that will create a connection between
automorphisms of $W$ and the graph algebras. Let $(J,W)$ be an
arbitrary graphon. We define a map $\psi_k:~\mathcal{Q}_k\rightarrow
L^\infty(J^k)$ by letting $\psi_k(G)$ be the $k$-variable function
$t_{x_1,x_2,\dots,x_k}(G,W)\in L^\infty(J^k)$. We extend this map
linearly to general quantum graphs. Note that $L^\infty(J^k)$ is a
commutative algebra with pointwise multiplication and addition, and
$\psi_k$ is an algebra homomorphism. The kernel of $\psi_k$ is equal
to $\mathcal{I}_k$ and thus the range of $\psi_k$ is isomorphic to
the $k$-th graph algebra of $W$. We denote by $\AA_k=\AA_k(J,W)$ this
subalgebra of $L^\infty(J^k)$.

We will need a subalgebra of $\AA_k$: let $\AA_k^0$ denote the linear
span of functions $\psi_k(G)$, where in $G\in\mathcal{G}_k^0$ (so its
labeled points are non-adjacent). By definition $\AA_1=\AA_1^0$.

\subsection{Metrics on graphons}

For two points $x,y$ of a graphon $(J,W)$, we define their {\it
neighborhood distance} by
\[
r_W(x,y) = \|W(x,.)-W(y,.)\|_1 = \int_J |W(x,z)-W(y,z)|\,dz.
\]
It may happen that $W(x,.)$ is not measurable for some $x$; however,
we can always change $W$ on a set of measure $0$ to make these
one-variable sections of it measurable. We will assume in the sequel
that these functions are measurable.

The distance function $r_W$ is not necessarily a metric, only a
semimetric, meaning that $r_W(x,y)$ may be $0$ for distinct points
$x$ and $y$. Such points are called {\it twins}. How to merge twins
to get a weakly isomorphic graphon for which $r_W$ is a metric, was
described in \cite{BCL} (see also \cite{Hombook}).

As a further step of ``purifying'' a graphon, we can replace the
metric space $(J,r_W)$ by its completion. Furthermore, in this new
topology the underlying probability measure may not have full
support; we may restrict the graphon to the support of the measure
(which is a closed and therefore complete subspace). This procedure
is described in \cite{LSz8}.

We call a graphon $(J,W)$ {\it pure}, if $(J,r_W)$ is a complete
metric space, and $\pi$ has full support (i.e., every open set has
positive measure). The procedure described above implies that every
graphon is weakly isomorphic to a pure graphon. Pure graphons will be
crucial in this paper, even in order to define automorphisms.

It will be sometimes convenient to use the $L^2$-distance instead of
the $L^1$-distance: we consider
\[
d_W(x,y) = \Bigl(\int\limits_J (W(x,z)-W(y,z))^2\,dz\Bigr)^{1/2}.
\]
Since trivially $d_W(x,y)^2\le r_W(x,y)\le d_W(x,y)$, these two
metrics define the same topology. In particular, the metric space
$(J,d_W)$ associated with a pure graphon $(J,W)$ is also complete and
the measure $\pi$ has full support.

One advantage of $d_W$ is that it can be expressed in terms of
restricted homomorphism densities. We consider the $2$-labeled
quantum graph $h$ in Figure \ref{FIG:H}. Then it is easy to check
that
\begin{equation}\label{EQ:DIST-HOM}
d_W(x,y)^2 =t_{xy}(h,W).
\end{equation}

\begin{figure}[htb]
  \centering
  \includegraphics*[height=80pt]{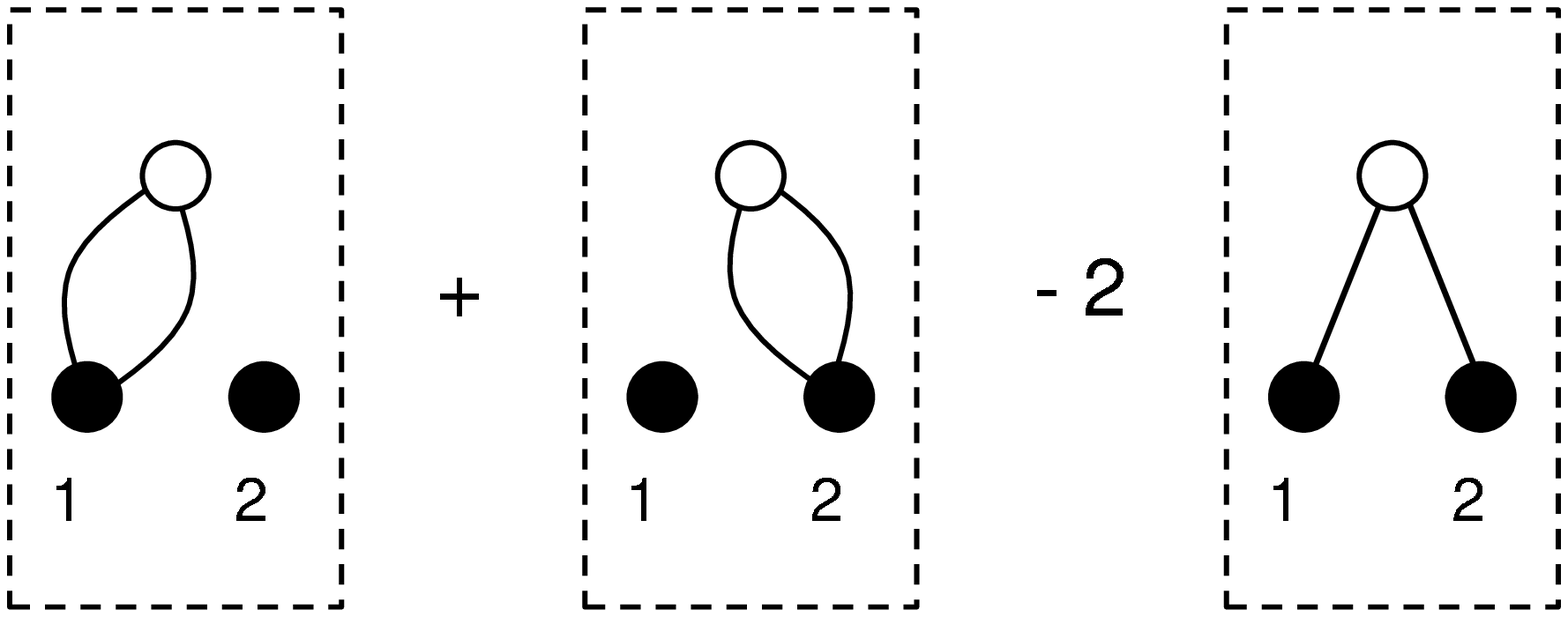}
  \caption{The quantum graph $h$ in the representation of the
  metric $d_W$.}\label{FIG:H}
\end{figure}

While the difference between the metrics $r_W$ and $d_W$ is not
essential, the $2$-neighborhood metric (called the {\it similarity
metric} in \cite{Hombook}) is more substantially different
\cite{LSz3,LSz8}. One way to define it is to introduce the ``operator
square'' of a graphon:
\[
(W\circ W)(x,y) = \int_J W(x,z)W(z,y)\,d\pi(x),
\]
and then consider the neighborhood distance of the graphon $(J,W\circ
W)$:
\[
\overline{r}_W(x,y)=r_{W\circ W}(x,y)
= \int_J \Bigl|\int_J (W(x,u)-W(y,u))W(u,z)\,d\pi(u)\Bigr|\,d\pi(z).
\]
This definition looks artificial, but in fact it has many nice
properties. It is easy to see that $\overline{r}(W)\le r_W$. If
$(J,W)$ is a pure graphon, then $(J, \overline{r}_W)$ is a metric
space (in particular, the distance between distinct points is
positive), which is not necessarily complete, but we can consider its
completion $(\overline{J},\overline{r}_W)$. We can extend the
probability measure to $\overline{J}$ by defining it to be $0$ on the
set of new points. We can also extend the function $W$ to
$\overline{W}:~\overline{J}\times\overline{J}\to[0,1]$ so that
$(\overline{J},\overline{W})$ is a graphon, and the metric
$\overline{r}_{\overline{W}}$ is equal to the completion of the
metric $\overline{r}_W$ (this takes some care). We will not
distinguish $\overline{r}_{\overline{W}}$ and $\overline{r}_W$ in the
sequel. On the other hand the metric $r_{\overline{W}}$ is quite
different: In terms of the $\overline{r}_W$ metric, all open sets
have positive measure, while the set $\overline{J}\setminus J$ of new
points is closed and has measure $0$. On the other hand, in terms of
the $r_{\overline{W}}$ metric, the set $\overline{J}\setminus J$ is
open (of measure zero).

The main property of this completion, which we will need, is that the
space $(\overline{J},\overline{r}_W)$ is compact (\cite{LSz8}; see
also \cite{Hombook}, Corollary 13.28). The metric $\overline{r}_W$
has another important property (\cite{Hombook}, Theorem 13.27):

\begin{prop}\label{PROP:WEAKWW}
If $(J,W)$ is pure graphon, then the metric $\overline{r}_W $ defines
exactly the weak topology on $\overline{J}$. In other words,
$\overline{r}_W(x_n,x)\to0$ $(x,x_n\in \overline{J})$ if and only if
\[
\int\limits_{\overline{J}} \overline{W}(x_n,y) f(y)\,dy\to
\int\limits_{\overline{J}} \overline{W}(x,y) f(y)\,dy
\]
for every bounded measurable function $f:~J\to\R$.
\end{prop}

A further important property of the metric $\overline{r}_W$, which we
don't use in this paper but is worth mentioning, is that a
decomposition of $\overline{J}$ into sets with small
$\overline{r}_W$-diameter corresponds to a (weak) regularity
partition. We refer to \cite{LSz3,LSz8,Hombook} for the exact
statement of this correspondence.

\subsection{Continuity of restricted homomorphism numbers}

We start with citing Lemma 13.19 from \cite{Hombook}:

\begin{lemma}\label{LEM:CONTIN}
Let $(J,W)$ be a pure graphon and let $F=(V,E)$ be a $k$-labeled
multigraph with nonadjacent labeled nodes. Then
\begin{equation}\label{EQ:RW-CONT}
|t_{x_1,\dots,x_k}(F,W)-t_{y_1,\dots,y_k}(F,W)|\le |E(F)|
\max_{i\le k} r_W(x_i,y_i).
\end{equation}
for all $x_1,\dots,x_k,y_1,\dots,y_k\in J$.
\end{lemma}

We need a version of this lemma for the $\overline{r}_W$-distance
instead of the $r_W$-distance. Some special cases of this were proved
in \cite{Hombook}, Section 13.4.

\begin{lemma}\label{LEM:HOM-CONT}
Let $(J,W)$ be a pure graphon, and let $F=(V,E)$ be a $k$-labeled
simple graph with nonadjacent labeled nodes. Then the restricted
homomorphism function $t_{x_1\dots x_k}(F,W)$ is continuous in each
of its variables $x_i$ on the metric space
$(\overline{J},\overline{r}_W)$.
\end{lemma}

\begin{proof}
Consider any point $x=(x_1,\dots,x_k)\in \overline{J}^k$, and let
$y_1,y_2,\dots\in \overline{J}$ be such that
$\overline{r}_W(y_n,x_1)\to0$ if $n\to\infty$. We want to show that
\begin{equation}\label{EQ:GOAL}
t_{y_nx_2\dots x_k}(F,W)\to t_{x_1\dots x_k}(F,W) \quad(m\to\infty).
\end{equation}

Let $N(1)=\{k+1,\dots,k+r\}$, and let $F'$ be obtained from $F$ by
deleting node $1$ and labeling nodes $k+1,\dots,k+r$. Then
\[
t_{y_nx_2\dots x_k}(F,W) = \int\limits_{J^r}
\prod_{i=k+1}^{k+r} W(y_n,z_i)
t_{x_2\dots x_k z_{k+1}\dots z_{k+r}}(F',W)\,dz_{k+1}\dots dz_{k+r}.
\]
The condition $\overline{r}_W(y_n,x_1)\to0$ implies that $W(y_n,.)\to
W(x_1,.)$ weakly as $n\to\infty$ (\cite{Hombook}, Theorem 13.7). It
is easy to see that this implies that $\prod_{i=k+1}^{k+r} W(y_n,z_i)
\to \prod_{i=k+1}^{k+r} W(x_1,z_i)$ (weakly as a function of
$(z_{k+1},\dots,z_{k+r})$), which in turn implies that
\begin{align*}
t_{y_nx_2\dots x_k}(F,W) \to &\int\limits_{J^r}
\prod_{i=k+1}^{k+r} W(x_1,z_i)
t_{x_2\dots x_k z_{k+1}\dots z_{k+r}}(F',W)\,dz_{k+1}\dots dz_{k+r}\\
&= t_{x_1x_2\dots x_k}(F,W),
\end{align*}
as claimed.
\end{proof}

Let us discuss the restrictions in these lemmas. It is obvious that
these lemmas do not remain valid if we allow edges between labeled
nodes: for example, $W(x,y)=t_{xy}(K_2^{\bullet\bullet},W)$ itself is
not necessarily continuous. Lemma \ref{LEM:CONTIN} implies that
$t_{x_1,\dots,x_k}$ is continuous (even Lipschitz) in the
neighborhood distance, simultaneously in all variables. Lemma
\ref{LEM:HOM-CONT}, however, fails to hold in this stronger sense;
see Example \ref{EXA:SYM} below (adapted from \cite{Hombook}, Example
13.30). This example also shows that in Lemma \ref{LEM:HOM-CONT} we
have to restrict $F$ to simple graphs. Inequality \eqref{EQ:RW-CONT}
also shows that, for a fixed $F$, the difference
$|t_x(F,W)-t_y(F,W)|$ can be estimated by $\max_i r_W(x_i,y_i)$,
independently of $W$. Example \ref{EXA:NON-LIP} below shows that,
even in the case $k=1$, no such estimate can be given in terms of
$\overline{r}_W(x,y)$.

\begin{example}\label{EXA:SYM}
For $y\in[0,1)$, let $y=0.y_1y_2\dots$ be the binary expansion of
$y$. Define $U(x,y)=y_k$ for $0\le y\le 1$ and $2^{-k-1} \le x\le
2^{-k}$. Define $U(0,y)=1/2$ for all $y$. This function is not
symmetric, so we put it together with a reflected copy to get a
graphon:
\[
W(x,y)=
 \begin{cases}
  U(2x,2y-1), & \text{if $x\le1/2$ and $y\ge1/2$}, \\
  U(2y,2x-1), & \text{if $x\ge1/2$ and $y\le1/2$},\\
  0, & \text{otherwise}.
 \end{cases}
\]
Let $u_k\in[2^{-1}+2^{-k},2^{-1}+2^{-k-1})$, then (as noted in
\cite{Hombook}) the sequence $(u_1,u_2,\dots)$ converges to the point
$0$ in the metric $\overline{r}_W$. On the other hand, for the
$3$-node path labeled at both endpoints
\[
t_{u_n}(C_2^{\bullet},W)=t_{u_n,u_n}(P_3^{\bullet\bullet},W) = \frac14,
\]
but
\[
t_0(C_2^{\bullet},W)=t_{0,0}(P_3^{\bullet\bullet},W) = \frac18,
\]
showing that $t_x(C_2^{\bullet},W)$ is not continuous at $x=0$, and
that $t_{x,y}(P_2^{\bullet\bullet},W)$, as a function of $x$ and $y$,
is not continuous at $(0,0)$.
\end{example}

\begin{example}\label{EXA:NON-LIP}
Consider the weighted graph $H$ given by the matrix of edgeweights
\[
A=\begin{pmatrix}
1/4&1/2&0&1\\
1/2&1&1&0\\
0&1&0&0\\
1&0&0&0
\end{pmatrix}
\]
and the vector of nodeweights
\[
b=\begin{pmatrix}
2/3-\eps\\
1/3-\eps\\
\eps\\
\eps
\end{pmatrix}.
\]
This weighted graph $H$ can be considered as a pure graphon with a
$4$-point underlying space. Then $H\circ H=H'$ is the weighted graph
given by the matrix of edgeweights
\[
A'=A\,\diag(b)A = \begin{pmatrix}
1/8&1/4&2/9&2/9\\
1/4&1/2&1/9&2/9\\
1/6&1/3&0&0\\
1/6&1/3&0&0
\end{pmatrix}+O(\eps),
\]
and the same nodeweights as before. Let $a$ and $b$ be the last two
nodes, then $\overline{r}_H(a,b) =O(\eps)$, but
$t_a(K_2^\bullet,W_H)=1/3-\eps$ and $t_b(K_2^\bullet,W_H)= 2/3-\eps$.
\end{example}

We have seen that excluding edges between the labeled nodes is
essential in both previous lemmas. The next lemma expresses $W$ in
terms of restricted homomorphism numbers for graphs with nonadjacent
labeled nodes, and gives a (rather weak, but still useful) remedy for
this restriction. We consider two sequences of quantum graphs $f_n$
and $g_n$ (Figure \ref{FIG:GG}), and define
\begin{equation}\label{EQ:UN-DEF}
U_n(x,y) = \frac{t_{xy}(g_n,W)}{t_{x}(f_n,W)}.
\end{equation}

\begin{figure}[htb]
  \centering
  \includegraphics*[height=160pt]{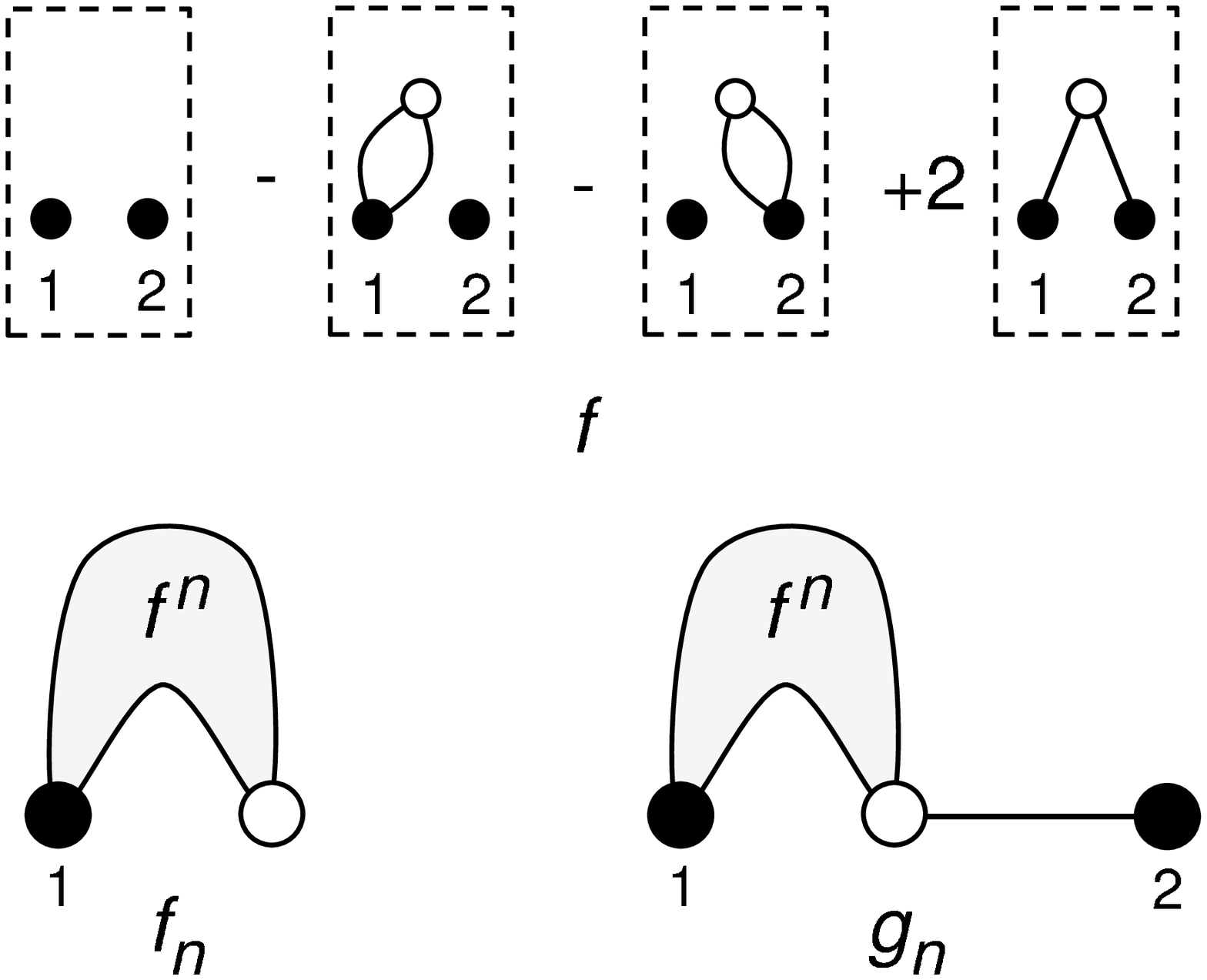}
  \caption{Quantum graphs $f$, $f_n$ and $g_n$. Here
  $f^n$ is obtained by gluing together $n$ copies of $f$
  along the labeled nodes; one of these nodes is then unlabeled
  (shown in white).}\label{FIG:GG}
\end{figure}

\begin{lemma}\label{LEM:FORMULA}
Let $(J,W)$ be a pure graphon, and let $U_n:~J^2\to[0,1]$ be defined
by \eqref{EQ:UN-DEF}. Then $\|U_n-W\|_1\to 0$ $(n\to\infty)$.
\end{lemma}

\begin{proof}
We have
\[
|U_n(x,y)-W(x,y)| =
\frac{|t_{xy}(g_n,W)-W(x,y)t_{x}(f_n,W)|}{t_{x}(f_n,W)}.
\]
Here the numerator can be expressed as
\begin{align*}
&|t_{xy}(g_n,W)-W(x,y)t_{x}(f_n,W)|
= \Bigl|\int\limits_J t_{xu}(f,W)^n(W(u,y)-W(x,y))\,du\Bigr|\\
&\le \int\limits_J t_{xu}(f,W)^n|W(u,y)-W(x,y)|\,du
= \int\limits_J (1-d_W(x,u)^2)^n|W(u,y)-W(x,y)|\,du,
\end{align*}
while the denominator is
\[
t_{x}(f_n,W) = \int\limits_J (1-d_W(x,y)^2)^n\,dy.
\]
Integrating over $y$, and using Cauchy--Schwarz, we get
\begin{align*}
\int_J |U_n(x,y)-W(x,y)| \,dy &\le
\frac{\int_J (1-d_W(x,u)^2)^nr_W(x,y)\,du}{\int_J (1-d_W(x,y)^2)^n\,du}\\&\le
\left(\frac{\int_J (1-d_W(x,u)^2)^nd_W(x,u)\,du}{\int_J (1-d_W(x,y)^2)^n\,du}\right)^{1/2}
\end{align*}
It is not hard to see that the right hand side tends to $0$ as
$n\to\infty$, which implies the lemma (in fact, a little more:
$U_n(x,.)\to W(x,.)$ in the $L^1$ metric for every $x$).
\end{proof}

\section{Compactness of the automorphism group}

\subsection{Automorphisms of graphons}

It only makes sense to define automorphisms of pure graphons.

Of course, one could define an ``automorphism'' of any graphon
$(J,W)$ as an invertible measure preserving map $\sigma:~J\to J$ such
that $W(x^\sigma,y^\sigma)=W(x,y)$ for almost all $x,y\in J$.
However, there is a lot of trouble with this notion: weakly
isomorphic graphons will have wildly different automorphism groups.
An example with many automorphisms is a stepfunction $W$: here
$\Aut(W)$ contains the group of all invertible measure preserving
transformations that leave the steps invariant (in addition to all
the automorphisms of the corresponding weighted graph). Note,
however, that if we purify a stepfunction, then we get a finite
weighted graph, and so the large and ``ugly'' subgroups consisting of
measure preserving transformations of the steps disappear. Another
problem would be that any permutation of points of a zero-measure set
should be considered an automorphism, so every graphon would have a
transitive automorphism group.

\begin{definition}\label{DEF:AUTO}
Let $(J,W)$ be a pure graphon. A measure preserving bijection
$\sigma:~J\to J$ is called an {\it automorphism} of $(J,W)$ if, for
every $x\in J$, the equality $W(x^\sigma,y^\sigma)=W(x,y)$ holds for
almost all $y\in J$.
\end{definition}

Note the change in the phrasing of the last condition: it is stronger
that requiring that $W(x^\sigma,y^\sigma)=W(x,y)$ for almost all
$x,y\in J$. This modification will exclude ``automorphisms'' like
interchanging two points.

(The simpler but inadequate definition is given in \cite{Hombook};
the results announced there hold true with the definition given
here.)

It is clear that every automorphism preserves the distances $r_W$ and
$\overline{r}_W$, and hence it extends to an automorphism of
$(\overline{J},\overline{W})$. The points of $\overline{J}\setminus
J$ can be identified in the graphon $(\overline{J},\overline{W})$ by
the property that every $r_{\overline{W}}$-neighborhood of them has
positive measure. So the automorphism groups of a pure graphon
$(J,W)$ and its completion $(\overline{J},\overline{W})$ are
essentially the same. In this section, we will mostly work with
$(\overline{J},\overline{W})$.

We can endow $\Aut(W)$ with a metric (and through this, with a
topology) by
\[
d(\sigma,\tau)=\sup_{x\in \overline{J}}
\overline{r}_W(x^\sigma,x^\tau).
\]

Not every isometry of the metric space $(J,r_W)$ (or of the metric
space $(\overline{J},\overline{r}_W)$) is an automorphism.

\begin{example}\label{EXA:ISOM-NOT}
Let $([0,1],W)$ be the pure graphon $W(x,y)=xy$, and consider the
direct sum $([0,1],W)\oplus ([0,1],1-W)$. This is pure as well, and
interchanging the two components is an isometry but not an
automorphism in general.
\end{example}

The following technical lemma shows that a slight apparent weakening
of the second condition in the definition of an automorphism leads to
the same concept. We will formulate it for the
$\overline{r}_W$-metric; for the $r_W$-metric the proof is similar
(in fact, much simpler).

\begin{lemma}\label{LEM:ISO-NULL}
Let $(J,W)$ be a pure graphon, and let $\phi:~\overline{J}\to
\overline{J}$ be a bijective measure preserving map that is an
isometry of $(\overline{J},\overline{r}_W)$ and satisfies
$\overline{W}^\phi=\overline{W}$ almost everywhere. Then $\phi$ is an
automorphism.
\end{lemma}

\begin{proof}
Let us call a point $x\in\overline{J}$ {\it nice}, if
$\overline{W}(x,y)=\overline{W}(x^\phi,y^\phi)$ for almost all $y\in
J$. The condition that $\overline{W}^\phi=\overline{W}$ almost
everywhere implies that almost all points are nice, but we want to
show that {\it all} points are nice.

To this end, let us fix $x\in J$. Since the measure has full support
in $(\overline{J},\overline{r}_W)$, every neighborhood of $x$ has
positive measure, and hence there is a sequence of nice points $x_n$
such that $\overline{r}_W(x_n,x)\to 0$. This means that
\begin{equation}\label{EQ:X2X}
\int\limits_J \Bigl| \int\limits_J
(\overline{W}(x,y)-\overline{W}(x_n,y))\overline{W}(y,z)\,dy\Bigr|\,dz
\to 0 \quad (n\to\infty).
\end{equation}
Also, since $\phi$ is an isometry,
\begin{equation*}\label{EQ:PHIX2PHIX}
\int\limits_J \Bigl| \int\limits_J (\overline{W}(x^\phi,y)
-\overline{W}(x_n^\phi,y))\overline{W}(y,z)\,dy\Bigr|\,dz
\to 0 \quad (n\to\infty).
\end{equation*}
Since $\phi$ is measure preserving, we can replace $y$ by $y^\phi$
and $z$ by $z^\phi$ in this equation:
\begin{equation*}\label{EQ:PHIX2PHIY}
\int\limits_J \Bigl| \int\limits_J (\overline{W}(x^\phi,y^\phi)
-\overline{W}(x_n^\phi,y^\phi))\overline{W}(y^\phi,z^\phi)\,dy\Bigr|\,dz
\to 0 \quad (n\to\infty).
\end{equation*}
Since the points $x_n$ are nice,
$\overline{W}(x_n^\phi,y^\phi)=\overline{W}(x_n,y)$ for almost all
$y$, and similarly $\overline{W}(y^\phi,z^\phi)=\overline{W}(y,z)$
for almost all pairs $(y,z)$. This implies that
\begin{equation*}\label{EQ:PHIX2YZ}
\int\limits_J \Bigl| \int\limits_J (\overline{W}(x^\phi,y^\phi)
-\overline{W}(x_n,y))\overline{W}(y,z)\,dy\Bigr|\,dz
\to 0 \quad (n\to\infty).
\end{equation*}
Comparing with \eqref{EQ:X2X}, we get
\begin{equation*}\label{EQ:PHIX2YZZ}
\int\limits_J \Bigl| \int\limits_J (\overline{W}(x^\phi,y^\phi)
-\overline{W}(x,y))\overline{W}(y,z)\,dy\Bigr|\,dz
\to 0 \quad (n\to\infty).
\end{equation*}
The left hand side does not depend on $n$, and hence it follows that
\begin{equation*}\label{Z0}
\int\limits_J (\overline{W}(x^\phi,y^\phi)
-\overline{W}(x,y))\overline{W}(y,z)\,dy=0
\end{equation*}
for almost all $z$. We can choose a sequence $z_n$ for which this
holds and for which $\overline{r}_{\overline{W}}(z_n,x)\to 0$. It is
easy to see that this implies that
\begin{equation*}\label{Z1}
\int\limits_J (\overline{W}(x^\phi,y^\phi)
-\overline{W}(x,y))\overline{W}(y,x)\,dy=0.
\end{equation*}
A similar argument gives
\begin{equation*}\label{Z2}
\int\limits_J (\overline{W}(x^\phi,y^\phi)
-\overline{W}(x,y))\overline{W}(y^\phi,x^\phi)\,dy=0.
\end{equation*}
Subtracting, we get
\[
\int\limits_J (\overline{W}(x^\phi,y^\phi)
-\overline{W}(x,y))^2\,dy=0,
\]
which implies that $\overline{W}(x^\phi,y^\phi) =\overline{W}(x,y)$
for almost all $y$. This proves the lemma.
\end{proof}

\subsection{Compactness}

The following fact is stated (without proof) in Section 13.5 of
\cite{Hombook}.

\begin{theorem}\label{THM:AUT-COMP}
The automorphism group of a pure graphon is compact.
\end{theorem}

This theorem is an immediate consequence of the following fact.

\begin{lemma}\label{LEM:CLOSED}
The automorphisms of a pure graphon $(J,W)$ form a closed subgroup of
the isometry group of $(\overline{J},\overline{r}_W )$.
\end{lemma}

\begin{proof}
Clearly every automorphism of $(J,W)$ is an isometry of
$(\overline{J},\overline{r}_W)$, and these isometries form a
subgroup. We want to prove that this subgroup is closed in the
topology of pointwise convergence.

Let $(\phi_n)$ be a sequence of automorphisms of $(J,W)$, and assume
that they converge to an isometry $\phi$. We want to prove that
$\phi$ is not only an isometry, but an automorphism. By Lemma
\ref{LEM:ISO-NULL}, it suffices to prove the following claims.

\begin{claim}\label{CLAIM:OPEN1}
For every open set X, $\pi(\phi(X)\triangle \phi_n(X))\to0$ as
$n\to\infty$.
\end{claim}

Indeed, since $\phi_n(x)\to \phi(x)$ for every $x\in J$, it follows
that for every $x\in X$, $\phi_n(x)\in\phi(X)$ if $n$ is large
enough. This means that every point belongs to a finite number of
sets $X\setminus \phi_n^{-1}(\phi(X))$ only, which implies that
$\pi(X\setminus \phi_n^{-1}(\phi(X)))=\pi(\phi_n(X)\setminus
\phi(X))\to0$. By a similar argument,
$\pi(\phi_n^{-1}(\phi(X))\setminus X)= \pi(\phi(X)\setminus
\phi_n(X))\to 0$. This implies the Claim.

\begin{claim}\label{CLAIM:CLOSED1}
The map $\phi$ is measure preserving.
\end{claim}

It suffices to show that $\phi$ preserves the measure of any open set
$X\subseteq\overline{J}$. By Claim \ref{CLAIM:OPEN1},
$\pi(\phi_n(X))\to\pi(\phi(X))$ as $n\to\infty$. Since $\phi_n$ is
measure preserving, this implies that $\pi(X)=\pi(\phi(X))$.

\begin{claim}\label{CLAIM:CLOSED2}
$\overline{W}^\phi=\overline{W}$ almost everywhere.
\end{claim}

It suffices to prove that for any two open sets $A$ and $B$,
\begin{equation}\label{EQ:GOAL1}
\int\limits_{A\times B} \overline{W}(x,y)\,dx\,dy
=\int\limits_{A\times B} \overline{W}(\phi(x),\phi(y))\,dx\,dy.
\end{equation}
For every $y\in J$,
\begin{align}\label{EQ:DIFF2}
\Bigl|\int\limits_A &\overline{W}(\phi_n(x),\phi_n(y))\,dx- \int\limits_A
\overline{W}(\phi(x),\phi_n(y))dx\Bigr|\\
&= \Bigl|\int\limits_{\phi_n(A)} \overline{W}(x,\phi_n(y))\,dx - \int\limits_{\phi(A)}
\overline{W}(x,\phi_n(y))dx\Bigr|
\le \pi(\phi(A)\triangle \phi_n(A)).\nonumber
\end{align}
Using that the maps $\phi_n$ are automorphisms,
\begin{align*}
\int\limits_{A\times B} &\overline{W}(x,y)\,dx\,dy
=\int\limits_{A\times B} \overline{W}(\phi_n(x),\phi_n(y))\,dx\,dy
=\int\limits_{A\times B} \overline{W}(\phi(x),\phi_n(y))\,dx\,dy\\
&~~~+\int\limits_{A\times B} \overline{W}(\phi_n(x),\phi_n(y))\,dx\,dy-
\int\limits_{A\times B}\overline{W}(\phi(x),\phi_n(y))\,dx\,dy
\end{align*}
The first term on the right side tends to $\int_{A\times B}
\overline{W}(\phi(x),\phi(y))\,dx\,dy$ by Proposition
\ref{PROP:WEAKWW}, and the difference in the last line tends to $0$
as $n\to\infty$ by \eqref{EQ:DIFF2} and Claim \ref{CLAIM:OPEN1}. This
proves Claim \ref{CLAIM:CLOSED2}, and thereby the Lemma.
\end{proof}

\section{Spectra}

\subsection{Spectral decomposition}

Since $W$ is bounded, the operator $T_W$ is Hilbert-Schmidt and hence
it has a spectral decomposition
\begin{equation}\label{EQ:SPEC1}
W(x,y)\sim\sum_{r=1}^\infty \lambda_r f_r(x)f_r(y),
\end{equation}
where the $\lambda_r$ are its nonzero eigenvalues and the functions
$f_r\in L^2(J)$ are the corresponding eigenfunctions, forming an
orthonormal system. Here $\lambda_r\to0$. By definition
\begin{equation}\label{EQ:EIGEN}
\lambda_rf_r(x)=\int\limits_J W(x,y)f_r(y)\,dy
\end{equation}
almost everywhere. We assume that $W(x,.)$ is measurable for every
$x$, and we can change $f_r$ on a set of measure $0$ so that
\eqref{EQ:EIGEN} holds for every $x\in J$. We note that
\eqref{EQ:EIGEN} implies that $f_r$ is bounded:
\[
|f_r(x)|=\frac1{|\lambda_r|}
\Bigl|\int\limits_J W(x,y)f_r(y)\,dy\Bigr| \le
\frac1{|\lambda_r|} \int\limits_J |f_r(y)|\,dy =
\frac{\|f_r\|_1}{|\lambda_r|} \le \frac{\|f_r\|_2}{|\lambda_r|}
=\frac1{|\lambda_r|}.
\]

We need the following simple observation: for every $x\in J$,
\begin{equation}\label{EQ:SUMSMALL0}
\sum_{r=1}^\infty \lambda_r^2 f_r(x)^2 = \|W(x,.)\|_2^2 = t_x((K_2^\bullet)^2,W).
\end{equation}
Indeed, using \eqref{EQ:EIGEN} and the fact that $\{f_r\}$ is an
orthonormal system, we get
\[
\sum_{r=1}^\infty \lambda_r^2 f_r(x)^2 =
\sum_{r=1}^\infty \Bigl(\int\limits_J W(x,y)
f_r(y)\Bigr)^2 \,dy =\sum_{r=1}^N \langle W(x,.), f_r\rangle^2
=\|W(x,.)\|_2^2.
\]
(the last equality follows because even though $\{f_r\}$ may not be a
complete orthogonal system, it can be extended by functions in the
nullspace of $T_W$ to such a system, and these additional functions
contribute $0$ terms). The second equality in \eqref{EQ:SUMSMALL0} is
trivial by definition. \eqref{EQ:SUMSMALL0} in turn implies that
\begin{equation}\label{EQ:SUMSMALL}
\sum_{r=1}^N \lambda_r^2 f_r(x)^2 \le 1.
\end{equation}

Expansion \eqref{EQ:SPEC1} may not hold pointwise, only in $L^2$; but
it follows from basic results on Hilbert-Schmidt operators that if we
take the inner product with any function $U\in L^2(J\times J)$, then
we get an equation:
\begin{equation}\label{EQ:EASY-CONV}
\int\limits_{J\times J} W(x,y)U(x,y)\,dx\,dy = \sum_{r=1}^\infty
\lambda_r \int\limits_{J\times J} f_r(x)f_r(y)U(x,y)\,dx\,dy,
\end{equation}
where the sum on the right side is absolutely convergent. We need the
following stronger fact:

\begin{lemma}\label{LEM:MORE-CONV}
Let $(J,W)$ be a graphon, and let \eqref{EQ:SPEC} be its spectral
decomposition.

\smallskip

{\rm(a)} For $U\in L^2(J)$ and $y\in J$, the sum
\begin{equation}\label{EQ:SUM1}
\sum_{r=1}^\infty \lambda_r f_r(y) \int\limits_J U(x)f_r(x) \,dx
\end{equation}
is absolutely convergent.

\smallskip

{\rm(b)} For every bounded measurable function $U:~J\times J\to\R$
and for almost all $y\in J$,
\begin{equation}\label{EQ:SUBS}
\int\limits_J W(x,y)U(x,y)\,dx = \sum_{r=1}^\infty
\lambda_r f_r(y) \int\limits_J U(x,y)f_r(x) \,dx.
\end{equation}
\end{lemma}

\begin{proof}
(a) We have
\begin{align}\label{EQ:SUMBOUND}
\sum_{r=N}^\infty \Bigl|&\lambda_r f_r(y) \int\limits_J U(x)f_r(x) \,dx\Bigr|
\le \sum_{r=N}^\infty
|\lambda_r| |f_r(y)| \Bigl|\int\limits_J U(x)f_r(x) \,dx\Bigr|\nonumber\\
&\le \Bigl(\sum_{r=N}^\infty \lambda_r^2
f_r(y)^2\Bigr)^{1/2}\Bigr(\sum_{r=N}^\infty \Bigl(\int\limits_J U(x)f_r(x)
\,dx\Bigr)^2\Bigr)^{1/2}.
\end{align}
Here the first factor is the tail of a convergent sum by
\eqref{EQ:SUMSMALL}, and hence it tends to $0$ as $n\to\infty$.
Furthermore, $\{f_r\}$ is an orthonormal system, and hence
\begin{gather*}
\sum_{r=N}^\infty \Bigl(\int\limits_J U(x)f_r(x) \,dx\Bigr)^2 \le
\sum_{r=1}^\infty \Bigl(\int\limits_J U(x)f_r(x) \,dx\Bigr)^2\\
=\sum_{r=1}^\infty \langle U, f_r \rangle^2 \le \|U\|_2^2,
\end{gather*}
proving (a).

\smallskip

Let $g_1(y)$ and $g_2(y)$ be the functions on the left and right
sides of equation \eqref{EQ:SUBS}. Then for any bounded measurable
function $h:~J\to \R$,
\begin{align*}
\langle h, g_1\rangle &= \int\limits_{J\times J} W(x,y)U(x,y)h(y)\,dx\,dy\\
&=\sum_{r=1}^\infty\lambda_r\int\limits_{J\times J} U(x,y)h(y)f_r(x)f_r(y)\,dx\,dy\\
&=\int\limits_J h(y) \sum_{r=1}^\infty \lambda_r f_r(y)
\int\limits_J U(x,y) f_r(x)\,dx\,dy= \langle h,g_2\rangle
\end{align*}
(where we use \eqref{EQ:EASY-CONV} and the fact that the sum in the
third line is absolutely convergent). This proves that $g_1=g_2$
almost everywhere.
\end{proof}

\subsection{Spectral decomposition of pure
graphons}\label{SEC:SPECPURE}

In this chapter we use the topological properties of pure graphons to
formulate finer statements about spectral decompositions. First of
all, note that if $(J,W)$ is a pure graphon then eigenfunctions of
$W$ are continuous functions on $\overline{J}$ in the metric
$\overline{r}_W$ (\cite{Hombook}, Corollary 13.29). Furthermore, the
eigenfunctions separate the points of $J$:

\begin{lemma}\label{specsepar}
If $(J,W)$ is a pure graphon, then for every pair of distinct points
$x,y\in J$ there is an eigenfunction $f$ of $W$ such that $f(x)\neq
f(y)$.
\end{lemma}

\begin{proof}
By way of contradiction, assume that $x$ and $y$ cannot be separated
this way. From $x\neq y$ we obtain that $\overline{r}_W(x,y)>0$, and
thus the functions $W\circ W(x,.)$ and $W\circ W(y,.)$ have a
positive distance in $L^2(J)$. On the other hand, $W\circ
W(z,.)=\sum_{i=1}^\infty \lambda_i^2 f_i(z)f_i(.)$ holds for every
fixed $z\in J$ where the sum is $L^2$-convergent. Applying this
formula for $z=x$ and $z=y$ together with our assumption that
$f_i(x)=f_i(y)$, we get a contradiction.
\end{proof}

\begin{lemma}\label{LEM:UNIF-EIGEN}
If $(J,W)$ is a pure graphon, then the sum on the left side of
\eqref{EQ:SUMSMALL0} converges uniformly for $x\in\overline{J}$.
\end{lemma}

\begin{proof}
Using continuity of the eigenfunctions we obtain that
every term on the left side of \eqref{EQ:SUMSMALL0} is continuous in
$\overline{r}_W$, and so is the right side by Lemma
\ref{LEM:HOM-CONT}. Since every term on the left side is nonnegative,
it follows by Dini's Theorem that the convergence is uniform in $x$.
\end{proof}

This allows us to get the following stronger version of Lemma
\ref{LEM:MORE-CONV} for pure graphons:

\begin{lemma}\label{LEM:UNIF-CONV}
{\rm(a)} If $(J,W)$ is a pure graphon, then the sum \eqref{EQ:SUM1}
is uniformly absolute convergent for $y\in\overline{J}$.

\smallskip

{\rm(b)} If, in addition, $U(x,y)$ is a continuous function of $y$
for every $x\in J$ in the neighborhood distance, then the expansion
\eqref{EQ:SUBS} holds for every $y\in J$.
\end{lemma}

\begin{proof}
(a) By Lemma \ref{LEM:UNIF-EIGEN},
\begin{equation}\label{EQ:UNIFORM-CONV}
\sup_x \sum_{r=N}^\infty \lambda_r^2 f_r(x)^2 \to 0\quad(N\to\infty).
\end{equation}
Hence the computation in \eqref{EQ:SUMBOUND} gives an estimate of the
tail uniformly for all $y\in\overline{J}$.

(b) The left side of \eqref{EQ:SUBS} defines a continuous function of
$y\in J$ in the metric $r_W$. Every term on the right side is also
continuous, and the convergence is uniform by the estimate
\eqref{EQ:SUMBOUND}, using \eqref{EQ:UNIFORM-CONV}. Hence the limit
is a continuous function of $y\in J$. The space $J$ has the property
that every nonempty open set has positive measure. If two continuous
functions are equal almost everywhere on such a space, then they are
equal everywhere.
\end{proof}

\begin{corollary}\label{COR:TRANS-COMPACT}
If the automorphism group of a pure graphon $(J,W)$ is transitive on
$J$, then $(J,W)$ is compact and $\overline{J}=J$.
\end{corollary}

\begin{proof}
Let $x\in J$, then the orbit of $x$ is a continuous image of
$\Aut(J,W)$, and so it is compact in the metric $r_W$. If the
automorphism group is transitive on $J$, then this orbit is $J$, and
hence $(J,r_W)$ is compact. Since $\overline{r}_W\le r_W$, this
implies that $(J,\overline{r}_W)$ is compact, and since $J$ is dense
in $(\overline{J},\overline{r}_W)$, it follows that $J=\overline{J}$.
\end{proof}

We use our results above about spectra to describe a way, more
explicit than convergence in $L^2$, of the convergence of the
expansion \eqref{EQ:SPEC1}. For a graphon $(J,W)$ and $\lambda>0$, we
define the graphon $(J,[W]_\lambda)$ by the following partial sum of
(\ref{EQ:SPEC1}):
\begin{equation}\label{EQ:PARTSPEC}
[W]_\lambda(x,y)=\sum_{|\lambda_r|\geq\lambda} \lambda_r f_r(x)f_r(y).
\end{equation}
Note that this sum is finite. If $W$ has multiple eigenvalues, then
the terms $\lambda_r f_r(x)f_r(y)$ depend on the basis chosen in the
eigenspaces, but $[W]_\lambda$ does not depend on this basis. Let
\begin{equation}\label{EQ:UL-DEF}
U_\lambda=\bigoplus_{|\lambda_r|\geq\lambda}E_{\lambda_r},
\end{equation}
where $E_{\lambda_r}$ is the eigenspace of $W$ corresponding to
$\lambda_r$. Let $\Pi_\lambda$ denote the orthogonal projection of
$L^2(J)$ onto $U_\lambda$. Then $T_{[W]_\lambda}=T_W\Pi_\lambda$.
From the inequality $\sum_{i=1}^\infty \lambda_i^2\leq 1$, it follows
that the rank of $[W]_\lambda$ (the dimension of $U_\lambda$) is at
most $1/\lambda^2$.

Assume that the eigenvalues are ordered so that
$|\lambda_1|\geq|\lambda_2|\geq\dots$. Let $\mu_\lambda$ denote the
probability distribution of the vector
$(f_1(x),f_2(x),\dots,f_d(x))\in\mathbb{R}^d$, where
$d=\dim(U_\lambda)$ and $x\in J$ is chosen randomly, and let
$S_\lambda\subset\mathbb{R}^d$ be the support of $\mu_\lambda$. Then
the purification of $(J,[W]_\lambda)$ can be defined as
$(S_\lambda,W'_\lambda)$, where
\begin{equation}\label{scalgraph}
W'_\lambda((x_1,x_2,\dots,x_d),(y_1,y_2,\dots,y_d))
=\sum_{i=1}^d \lambda_i x_iy_i.
\end{equation}
A coordinate-independent way of describing $\mu_\lambda$ is to
consider the dual space of $U_\lambda$. For each $x\in J$, we
consider the linear functional $f\mapsto (T_W f)(x)$ ($f\in
U_\lambda$). If $x\in J$ is chosen randomly we obtain the probability
distribution $\mu_\lambda$ on $U_\lambda^*$, and we can define
$S_\lambda$ as its support. We will need the next lemma, which is a
direct consequence of the results in the paper \cite{Sz2}.

\begin{lemma}\label{meascon}
Let $\{U_n\}_{n=1}^\infty$ be a convergent sequence of graphons with
limit $W$. Assume that $\lambda>0$ is not an eigenvalue of $W$. Then
there is subsequence $\{W_n\}_{n=1}^\infty$ in $\{U_n\}_{n=1}^\infty$
and choices of orthonormal eigenvectors for $[W_n]_\lambda$ and
$[W_\lambda]$ such that the measures $\mu^n_\lambda$ constructed
above for $W_n$ converge to $\mu_\lambda$ weakly. \proofend
\end{lemma}

\ignore{
\begin{proof}
It is proved in \cite{Sz2} that $[W_n]_\lambda$ converges to
$[W]_\lambda$ in the so-called $\delta_1$ metric. From this, it is
easy to see that for every $\epsilon>0$, if $n$ is big enough, then
for any orthonormal basis for the eigenvectors of $[W_n]_\lambda$
there is an orthonormal basis for the eigenvectors of $[W]_\lambda$
that is element-by-element $\epsilon$-close in $L^2$. The compactness
of the set possible orthonormal bases completes the proof.
\end{proof}
}

If $\alpha<\beta$, then the projection
\[
P_{\alpha,\beta}:\mathbb{R}^{\dim(U_\alpha)}
\rightarrow\mathbb{R}^{\dim(U_\beta)}
\]
(by forgetting the last $\dim(U_\alpha)-\dim(U_\beta)$ coordinates)
transforms $\mu_\alpha$ into $\mu_\beta$. The map $P_{\alpha,\beta}$
is surjective from $S_\alpha$ to $S_\beta$.

Let $\{\alpha_i\}_{i=1}^\infty$ be a decreasing sequence tending to
$0$. Let $S$ be the inverse limit of the system
$\{P_{\alpha_{i+1},\alpha_i}:S_{\alpha_{i+1}}\rightarrow
S_{\alpha_i}\}_{i=1}^\infty$. This means that
\[
S=\{(s_1,s_2,\dots):~\forall i\in\mathbb{N},
s_i=P_{\alpha_{i+1},\alpha_i}s_{i+1}\}\subset \prod_{i=1}^\infty
S_{\alpha_i}.
\]
The limit of $\{\mu_{\alpha_i}\}_{i=1}^\infty$ defines a probability
measure $\mu$ on the compact set $S$. Let $(S,U_{\alpha_i})$ be the
graphon defined on $S$ using the formula \eqref{scalgraph} for the
$i$-th coordinate.

\begin{lemma}\label{specpure}
For every graphon $(J,W)$ there is a measure preserving homeomorphism
$\tau:\overline{J}\rightarrow S$ such that
$(U_{\alpha_i})^\tau=[\,\overline{W}\,]_{\alpha_i}$ holds for every
$i$.
\end{lemma}

\begin{proof}
Notice that the construction of $(S,U_{\alpha_i})$ depends only on
the weak isomorphism class of $W$, and so we can assume that $(J,W)$
is pure. The maps $\tau_i:~x\mapsto (f_1(x),f_2(x),\dots,f_d(x))$
from $\overline{J}$ to $S_i$ (where $d=\dim(U_{\alpha_i})$) are
continuous in the $\overline{r}_W$ metric. Hence the map
$\tau=(\tau_1,\tau_2,\dots):~\overline{J}\to S$ is also continuous.
Since $\tau$ separates elements in $\overline{J}$ (to see this, apply
lemma \ref{specsepar} for $W\circ W$), it is a bijection between
$\overline{J}$ and $S$. The desired property is clear from the
definition of $\tau$.
\end{proof}

\subsection{Subdividing edges}

As an application of spectral decomposition, we prove the following
generalization of Lemma 5.1 in \cite{BCL} (which will be needed later
on).

\begin{lemma}\label{LEM:ORBITS2}
Let $(J_1,W_1)$ and $(J_2,W_2)$ be two pure graphons and let $a\in
J_1^k$, $b\in J_2^k$. Let $h$ be a $k$-labeled quantum multigraph. In
every constituent of $h$, select an edge such that at least one
endpoint of it is unlabeled, and let $h_m$ denote the $k$-labeled
quantum multigraph obtained from $h$ by subdividing the selected edge
by $m-1$ new nodes in every constituent. Suppose there exists an
$m_0\ge 2$ such that $t_a(h_m,W_1)=t_b(h_m,W_2)$ for every $m\ge
m_0$. Then $t_a(h,W_1)=t_b(h,W_2)$.
\end{lemma}

\begin{proof}
Let $g_i$ be obtained from $h$ by keeping only those terms in which
one endpoint of the selected edge is labeled $i$ ($1\le i\le k$). Let
$g_0$ be the sum of the remaining terms, where the selected edge has
no labeled endpoint. Let $g_i'$ be the $(k+1)$-labeled quantum
multigraph obtained from $g_i$ by deleting the selected edge from
each constituent and labeling its unlabeled endpoint by $k+1$. Let
$g_0'$ be the $(k+2)$-labeled quantum multigraph obtained from $g_0$
by deleting the selected edge from each constituent and labeling its
endpoints by $k+1$ and $k+2$. Then
\begin{align*}
t_a(h_m, W_1)=&\sum_{i=1}^k
\int\limits_{J_1} W_1^{\circ m}(a_i,x)t_{ax}(g_i',W_1)\,dx\\
&+\int\limits_{J_1\times J_1} W_1^{\circ m}(x,y)t_{axy}(g_0',W_1)\,dx\,dy.
\end{align*}
We use the spectral decomposition
\begin{equation}\label{EQ:SPEC}
W_1^{\circ m}(x,y)\sim\sum_{r=1}^\infty \lambda_r^m f_r(x)f_r(y).
\end{equation}
This decomposition holds almost everywhere for $m\ge2$, but for
$m=1$, we can only claim that the sums on the right sides converge to
the function on the left in $L^2$. Since the graphon is pure, Lemma
\ref{LEM:UNIF-CONV} implies that the expansion
\begin{align*}
t_a(h_m, W_1)=& \sum_{i=1}^k\sum_{r=1}^\infty \lambda_r^m f_r(a_i)
\int\limits_{J_1} t_{ax}(g_i',W_1)f_r(x) \,dx\\
&+ \sum_{r=1}^\infty \lambda_r^m \int\limits_{J_1\times J_1}
f_r(x)f_r(y)t_{axy}(g_0',W_1)\,dx\,dy
\end{align*}
holds for all $m\ge 1$. We have an analogous expansion for $t_b(h_m,
W_2)$. If these two expressions are equal for every integer $m\ge
m_0$, then they are also equal for $m=1$ (see e.g. \cite{Hombook},
Proposition A.21).
\end{proof}

\begin{corollary}\label{COR:ORBITS2}
Let $(J_1,W_1)$ and $(J_2,W_2)$ be two pure graphons and let $a\in
J_1^k$, $b\in J_2^k$.

\smallskip

{\rm(a)} If
\begin{equation}\label{EQ:EQUAL2}
t_a(F,W_1)=t_b(F,W_2)
\end{equation}
for every $k$-labeled simple graph $F$, then \eqref{EQ:EQUAL2} holds
for every $k$-labeled multigraph $F$.

\smallskip

{\rm(b)} If \eqref{EQ:EQUAL2} holds for every $k$-labeled simple
graph $F$ with nonadjacent labeled nodes, then \eqref{EQ:EQUAL2}
holds for every $k$-labeled multigraph $F$ with nonadjacent labeled
nodes.
\end{corollary}

\begin{proof}
(b) follows from Lemma \ref{LEM:ORBITS2} by induction on the number
of parallel edges. To prove (a), it suffices to note that
$W_1(a_i,a_j)=W_2(b_i,b_j)$ follows by considering the simple graph
$F$ with a single edge connecting $i$ and $j$.
\end{proof}

\subsection{Automorphism groups and spectral decomposition}

Let $g:~J\rightarrow J$ be an automorphism of a graphon $(J,W)$.
Notice that if $f$ is an eigenfunction of length $1$ of $W$ then
$f^g$ is also an eigenfunction of length $1$ corresponding to the
same eigenvalue. As a consequence every automorphism of $W$ acts on
the space $U_\lambda$ defined in \eqref{EQ:UL-DEF} as an element in
$O_{\lambda}:=\bigoplus_{|\lambda_r|\ge \lambda}O(E_{\lambda_r})$
where $O(E_{\lambda_r})$ is the orthogonal group on $E_{\lambda_r}$.
The corresponding action on the dual space $U^*_\lambda$ leaves the
measure $\mu_\lambda$ invariant. We will denote by $\Gamma_\lambda$
the finite dimensional compact group formed by all elements
$O_{\lambda}$ that preserve $\mu_\lambda$. (Note that
$\Gamma_\lambda$ is the automorphism group of $[W]_\lambda$.)

The group $O_\alpha$ acts on both $U_\alpha$ and $U^*_\alpha$. Since
$U_\beta$ is an invariant subspace of $O_\alpha$, the group
$O_\alpha$ acts on $U^*_\beta$ as well. In particular, there is a
homomorphism $h_{\alpha,\beta}:~\Gamma_\alpha\to \Gamma_\beta$. We
denote by $\Gamma_W$ the inverse limit of the system
$\{h_{\alpha_{i+1},\alpha_i}\}_{i=1}^\infty$.

We can describe the automorphism group of a compact graphon using
representation of a graphon above.

\begin{lemma}\label{specaut}
For every graphon $(J,W)$ the action of $\Aut(W)$ on $\overline{J}$
can be obtained as $\tau^{-1}\circ\Gamma_W\circ\tau$, where $\tau$ is
the function in Lemma \ref{specpure}.
\end{lemma}

\begin{proof}
We may assume that the graphon $(J,W)$ is pure. First we show that
$\Aut(W)\subseteq \tau^{-1}\circ\Gamma_W\circ\tau$. Every
automorphism of $W$, restricted to $U_{\alpha_i}$ $(i=1,2,\dots)$,
induces a consistent sequence of elements in $\prod_{i=1}^\infty
\Gamma_{\alpha_i}$. It follows that
$\tau\circ\Aut(W)\circ\tau^{-1}\subseteq\Gamma_W$. The other
containment is a direct consequence of Lemma \ref{specpure}: elements
of $\tau^{-1}\circ\Gamma_W\circ\tau$ act on $\overline{J}$
continuously and leave $[\,\overline{W}\,]_{\alpha_i}$ invariant for
every $i$. This means that they also fix $\overline{W}$.
\end{proof}

\section{Orbits of the automorphism group}

\subsection{Characterization of the orbits}

The following theorem characterizes the orbits of the automorphism
group of a graphon.

\begin{theorem}\label{THM:ORBITS}
Let $(J,W)$ be a pure graphon, and let
$a_1,\dots,a_k,b_1,\dots,b_k\in J$. Then there exists an automorphism
$\phi\in\Aut(J,W)$ such that $a_i^\phi=b_i$ if and only if
$t_{a_1\dots a_k}(F,W)=t_{b_1\dots b_k}(F,W)$ for every $k$-labeled
simple graph $F$ in which the labeled nodes are independent.
\end{theorem}

The following version is more general (at least formally).

\begin{theorem}\label{THM:ORBITS2}
Let $(J_1,W_1)$ and $(J_2,W_2)$ be two pure graphons and let
$\alpha_i\in J_i^k$. Then there exists a measure preserving bijection
$\phi:~J_1\to J_2$ such that $W_2^\phi=W_1$ almost everywhere and
$\alpha_{1,i}^\phi=\alpha_{2,i}$ if and only if
$t_{\alpha_1}(F,W_1)=t_{\alpha_2}(F,W_2)$ for every $k$-labeled
simple graph $F$.
\end{theorem}

The proof of this theorem is a modification of the proof of the main
result of \cite{BCL}, combined with more recent methods involving
pure graphons.

First, we note that the condition in the theorem is self-sharpening:
by Corollary \ref{COR:ORBITS2}, the condition holds for every
$k$-labeled multigraph $F$. The following lemma is the main step in
the proof.

\begin{lemma}\label{LEM:COUPLE}
Let $(J_1,W_1)$ and $(J_2,W_2)$ be two graphons and let $a\in J_1^k$,
$b\in J_2^k$ such that
\[
t_a(F,W_1)=t_b(F,W_2)
\]
for every $k$-labeled multigraph $F$. Let $\pi_i$ denote the
probability measure of $J_i$. Then we can couple $\pi_1$ with $\pi_2$
so that if $(X,Y)$ is a pair from the coupling distribution, then
\[
t_{a_1\dots a_kX}(F,W_1)=t_{b_1\dots b_kY}(F,W_2)
\]
almost surely for every $(k+1)$-labeled multigraph $F$.
\end{lemma}

\begin{proof}
Consider two random points $X$ from $\pi$ and $Y$ from $\pi'$, and
the random variables
\[
A=(t_{a_1\dots a_kX}(F,H):~F\in\FF_{k+1})\quad
\text{and}\quad
B= (t_{b_1\dots b_kY}(F,H):~F\in\FF_{k+1})
\]
with values in $[0,1]^{\FF_{k+1}}$. We claim that the variables $A$
and $B$ have the same distribution. It suffices to show that $A$ and
$B$ have the same mixed moments. If $F_1,\dots,F_m\in\FF_{k+1}$, and
$q_1,\dots,q_m$ are nonnegative integers, then the corresponding
moment of $A$ is
\[
\E\Bigl(\prod_{i=1}^m t_{a_1\dots a_kX}(F_i,H)^{q_i}\Bigr)
=\E\bigl(t_{a_1\dots a_k X}(F_1^{q_1}\dots F_m^{q_m},H)\bigr)
=t_a(F,H),
\]
where the multigraph $F$ is obtained by unlabeling the node labeled
$k+1$ in the multigraph $F_1^{q_1}\dots F_m^{q_m}$. Expressing the
moments of $B$ in a similar way, we see that they are equal by
hypothesis. This proves that $A$ and $B$ have the same distribution.

Using Lemma 6.2 of \cite{BCL} it follows that we can couple the
variables $X$ and $Y$ so that $A=B$ with probability 1. In other
words,
\[
t_{a_1\dots a_kX}(F,H)= t_{b_1\dots b_kY}(F,H')
\]
for every $F\in\FF_{k+1}$ with probability 1.
\end{proof}

For an infinite sequence $X\in J^\N$, let $X[n]$ denote its prefix of
length $n$.

\begin{lemma}\label{LEM:COUPLE2}
Under the conditions of the previous lemma, we can couple $\pi_1^\N$
with $\pi_2^\N$ so that if $(X,Y)$ is a pair from the coupling
distribution, then for every $n\ge0$ and every $(k+n)$-labeled graph
$F$,
\[
t_{a_1\dots a_kX[n]}(F,W_1)=t_{b_1\dots b_kY[n]}(F,W_2)
\]
almost surely.
\end{lemma}

\begin{proof}
By Lemma \ref{LEM:COUPLE}, we can define recursively a coupling
$\kappa_n$ of $\pi_1^n$ with $\pi_2^n$ so that
$t_X(F,W_1)=t_Y(F,W_2)$ almost surely for every $F\in\FF_{k+n}$, and
$\kappa_{n+1}$, projected to the first $n$ coordinates in both
spaces, gives $\kappa_n$. The distributions $\kappa_n$ give a
distribution $\kappa$ on $J_1^\N\times J_2^\N$, which clearly has the
desired properties.
\end{proof}

The following lemma can be considered as a version of the theorem for
infinite sequences.

\begin{lemma}\label{LEM:MAPPING}
Let $(J_1,W_1)$ and $(J_2,W_2)$ be two pure graphons, and let
$a_i=(a_{i,1},a_{i,2},\dots)\in J_i^\N$ be a sequence whose elements
are dense in $J_i$. Suppose that $t_{a_1}(F,W_1) =t_{a_2}(F,W_2)$ for
every partially labeled multigraph $F$. Then there is a measure
preserving bijection $\phi:~J_1\to J_2$ such that $W_2^\phi=W_1$
almost everywhere and $a_{1,j}^\phi =a_{2,j}$ for all $j\in\N$.
\end{lemma}

\noindent{The notation $t_a(F,W)$, where $a$ is an infinite sequence,
means that only those elements of $a$ are considered whose subscript
occurs in $F$ as a label.}

\begin{proof}
We start with noticing that
\begin{equation}\label{EQ:DIST}
d_{W_1}(a_{1,i},a_{1,j})= d_{W_2}(a_{2,i},a_{2,j}).
\end{equation}
This follows by \eqref{EQ:DIST-HOM} and the hypothesis of the lemma.

For $x\in J_1$, take a subsequence $(a_{1,i_1},a_{1,i_2},\dots)$ such
that $a_{i,i_n}\to x$. Then $(a_{1,i_1},a_{1,i_2},\dots)$ is a Cauchy
sequence, and hence, by \eqref{EQ:DIST}, so is the sequence
$(a_{2,i_1},a_{2,i_2},\dots)$, and since $(J_2,r_{W_2})$ is complete,
it has a limit $x^\phi$. It is easy to see that this map is
well-defined (i.e., it does not depend on the choice of the sequence
$(a_{1,i_1},a_{1,i_2},\dots)$), and that $\phi$ is bijective.

Next, we claim that for every sequence $x_1,\dots,x_k\in J_1$ and
every multigraph $F$ with nonadjacent labeled nodes
\begin{equation}\label{EQ:TXPHI}
t_{x_1^\phi,\dots,x_k^\phi}(F,W_2) = t_{x_1,\dots,x_k}(F,W_1).
\end{equation}
Indeed, this holds if every $x_i$ is an element of the sequence $a_1$
by hypothesis, and then it follows for all $x_i$ by the continuity of
$t_{x_1,\dots,x_k}(F,W_1)$ (Lemma \ref{LEM:CONTIN}).

Finally, consider the function
\[
U_n(x,y) = \frac{t_{x,y}(g_n,W_1)}{t_{x}(g_n',W_1)}.
\]
By Lemma \ref{LEM:FORMULA}, $\|W_1-U_n\|_1\to0$ as $n\to\infty$.
Also, by \eqref{EQ:TXPHI},
\[
U_n(x,y) = \frac{t_{x^\phi,y^\phi}(g_n,W_2)}{t_{x^\phi}(g_n',W_2)},
\]
and applying Lemma \ref{LEM:FORMULA} again,
$\|W_2^\phi-U_n\|_1\to 0$ as $n\to\infty$. This implies that
$W_1=W_2^\phi$ almost everywhere.
\end{proof}

Now we are ready to prove the main theorem of this section.

\begin{proof*}{Theorem \ref{THM:ORBITS2}}
Let $X_1,X_2\dots$ be independent random points of $J_1$, and let
$Y_1,Y_2\dots$ be independent random points of $J_2$. Applying Lemma
\ref{LEM:COUPLE2} repeatedly, we can couple $X_1,X_2\dots$ with
$Y_1,Y_2\dots$ so that, for any $(k+r)$-labeled graph $F$,
\begin{equation}\label{EQ:TXY}
t_{a_1\dots a_kX_1\dots X_r}(F,W_1) = t_{b_1\dots b_kY_1\dots Y_r}(F,W_2).
\end{equation}

With probability $1$, the elements of both sequences
$a=(a_1,\dots,a_k,X_1,X_2,\dots)$ and
$b=(b_1,\dots,b_k,Y_1,Y_2,\dots)$ are dense in $J_1$ and $J_2$,
respectively. Let us fix such a choice, then by Lemma
\ref{LEM:MAPPING} there is a measure preserving bijection
$\phi:~J_1\to J_2$ such that $W_2^\phi=W_1$ almost everywhere and
$a_i^\phi =b_i$ for all $i\le k$. This proves the theorem.
\end{proof*}

\begin{corollary}\label{COR:ALGINV}
Let $(J,W)$ be a pure graphon. Then the closure of $\AA_k^0$ in
$L^\infty(\overline{J}{}^k)$ consists of all continuous
$\Aut(J,W)$-invariant functions on $(\overline{J},\overline{r}_W)^k$.
\end{corollary}

\begin{proof}
Lemma \ref{LEM:CONTIN} implies that all functions in $\AA_k^0$ are
continuous and clearly they are invariant under automorphisms. The
other containment follows from the Stone-Weierstrass theorem, since
$(\overline{J},\overline{r}_W)^k$ is compact, and by Theorem
\ref{THM:ORBITS} the elements in $\AA_k^0$ separate the orbits of
$\Aut(J,W)$.
\end{proof}

\subsection{Node-transitive graphons}

Let $\mathbb{G}$ be the automorphism group of the pure graphon
$(J,W)$. We consider the natural action of $\mathbb{G}$ on functions
on $J$ defined by $f^g(x)=f(x^g)$. Similarly $\Gbb$ acts diagonally
on functions on $J^n$. For a subset $S\subset L^\infty(J^n)$ we
denote by $S^\mathbb{G}$ the set of $\mathbb{G}$-invariant elements
in $S$. It is clear that restricted homomorphism functions are
invariant under the action of $\mathbb{G}$ and thus all the algebras
$\AA_k$ are $\mathbb{G}$-invariant.

\begin{definition}
A graphon is called {\it node-transitive} if the automorphism group
of its pure representation $(J,W)$ acts transitively on $J$.
\end{definition}

The next theorem gives an algebraic characterization of
node-transitive graphons.

\begin{theorem}\label{THM:ALGTR}
Let $(J,W)$ be a graphon. The following statements are equivalent.

\smallskip

{\rm(i)} $(J,W)$ is node-transitive.

\smallskip

{\rm(ii)} The functions $t_x(F,W)$ are essentially constant on $J$
for all $F\in\mathcal{G}_1$.

\smallskip

{\rm(iii)} $\dim(\AA_1)=1$.

\smallskip

{\rm(iv)} The first connection matrix $M_1$ of $W$ has rank $1$.

\smallskip

{\rm(v)}  $t(\ul{F^2},W)t(\ul{H^2},W)=t(\ul{FH},W)^2$ for all
$F,H\in\mathcal{G}_1$.
\end{theorem}

\begin{proof}
We may assume that $(J,W)$ is pure. If (i) holds, then $\mathbb{G}$
is transitive on $J$, and so every function $t_x(F,W)$ is constant on
$J$, which implies (ii). Conversely, (ii) implies by Theorem
\ref{THM:ORBITS} that $\Gbb$ is transitive on $J$, so (i) holds. Thus
(i) and (ii) are equivalent. Every constant function is in $\AA_1$,
hence $\dim(\AA_1)\ge1$, and so (ii) is equivalent to (iii). We know
that $\rk(M_1)=\dim(\AA_1)$, so (iv) is just a re-statement of (iii).
Finally, (v) is a re-statement of (iv), since $M_k$ is positive
semidefinite.
\end{proof}

Examples for node-transitive graphons are finite node-transitive
graphs. Other examples are graphons defined on compact topological
groups.

\begin{definition}\label{DEF:CAYLEY}
Let $\mathbb{G}$ be a second countable compact topological group,
which, together with its Haar measure, defines a standard probability
space. Let $f:~\mathbb{G}\rightarrow[0,1]$ be a measurable function
such that $f(x)=f(x^{-1})$. Then the graphon
$W:~\mathbb{G}\times\mathbb{G}\rightarrow[0,1]$ defined by
$W(x,y)=f(xy^{-1})$ is called a {\it Cayley graphon}.
\end{definition}

Note that the condition $f(x)=f(x^{-1})$ is needed to guarantee that
$W$ is symmetric. By omitting this condition we get ``directed Cayley
graphons''.

\begin{theorem}\label{THM:TRCAYLEY}
Cayley graphons are node-transitive. Conversely, every
node-transitive graphon is weakly isomorphic to a Cayley graphon.
\end{theorem}

Note that a finite node-transitive graph $G$ is not necessarily a
Cayley graph (for example, the Petersen graph). However one can
obtain a Cayley graph $G'$ from $G$ by replacing every vertex by $m$
vertices and every edge by a complete bipartite graph $K_{m,m}$. The
value $m$ is the size of the stabilizer of a vertex in $G$ in the
automorphism group. The graph $G'$ is weakly isomorphic to $G$ as a
graphon.

\begin{proof}
Let $(\Gbb,W)$ be a Cayley graphon on the compact topological group
$\Gbb$. It is clear that $\mathbb{G}$ acts transitively (with
multiplication from the right) on this graphon. (However, $W$ might
not be pure.) It follows that the restricted homomorphism functions
$t_x(F,W)$ are all constant on $\mathbb{G}$. The third condition in
Theorem \ref{THM:ALGTR} shows that $W$ is node-transitive.

To prove the second assertion, let $(J,\pi,W)$ be a node-transitive
graphon; we may assume that it is pure. Let $\Gbb$ be its
automorphism group. We know that $\Gbb$ is compact, and so it has a
normalized Haar measure $\mu$. Let us fix an element $c\in J$, and
define the function $U:~\mathbb{G}\times\mathbb{G}\rightarrow[0,1]$
by $U(g,h)=W(c^g,c^h)$. We claim that $(\Gbb,\mu,U)$ is a Cayley
graphon weakly isomorphic to $(J,\pi,W)$.

\begin{claim}\label{CLAIM:MP}
The map $\alpha:~g\mapsto c^g$ defined on $\mathbb{G}$ is measure
preserving.
\end{claim}

The definition of the metric on $\Aut(W)$ implies that
$\overline{r}_W(\alpha(g),\alpha(h))=\overline{r}_W(c^g,c^h)\le
d(g,h)$ for every $g,h\in\Gbb$. This shows that $\alpha$ is
continuous and hence, measurable.

Let $\mathcal{B}$ denote the sigma-algebra on $\mathbb{G}$ formed by
the sets $\alpha^{-1}(X)$, where $X$ is a Borel set in $J$, and let
$\nu$ denote the measure on $\mathcal{B}$ that is the pullback of
$\pi$. It is clear that $\nu$ is $\mathbb{G}$-invariant. Standard
topological group theory shows that $\nu$ extends to the Borel
sigma-algebra on $\mathbb{G}$ as the normalized Haar measure $\mu$.
This proves the Claim.

Since by definition $U=W^\alpha$, the Claim implies that $(\Gbb,U)$
is weakly isomorphic to $(J,W)$. Let $f:~\mathbb{G}\to [0,1]$ be
defined by $f(g)=W(c^g,c)$. Then
$U(g,h)=W(c^g,c^h)=w(c^{gh^{-1}},c)=f(gh^{-1})$, so $(\Gbb,U)$ is a
Cayley graphon.
\end{proof}

\begin{remark}\label{REM:GROUPS}
Theorem \ref{THM:TRCAYLEY} creates a connection between graph limit
theory and an interesting and rich limit theory for functions on
groups (see \cite{Sz1}, \cite{Sz2}). The idea is the following. Let
$\{f_i:\mathbb{G}_i\rightarrow [0,1]\}_{i=1}^\infty$ be a sequence of
measurable functions on compact groups. We say that the sequence
$f_i$ is convergent if the corresponding Cayley graphons
$\{W_i\}_{i=1}^\infty$ converge. By Proposition \ref{PROP:TRANS-LIM}
and Theorem \ref{THM:TRCAYLEY}, we prove that the limit of
$\{W_i\}_{i=1}^\infty$ is weakly isomorphic to a Cayley graphon
defined by a measurable function $f:\mathbb{G}\rightarrow[0,1]$ on a
compact group. We say that $f$ is the {\it limit object} of the
sequence $\{f_i\}_{i=1}^\infty$. It turns out that one can define
this limit concept without passing to graphons. This point of view
was heavily used in the second author's approach \cite{Sz1} to higher
order Fourier analysis.
\end{remark}

\subsection{Limits of node-transitive graphons}

We start with the observation that, if a convergent graph sequence
consists of node-transitive graphs, then their limit graphon is
node-transitive as well. More generally, we have the following
consequence of the fifth condition in Theorem \ref{THM:ALGTR}.

\begin{corollary}\label{PROP:TRANS-LIM}
If a sequence of node-transitive graphons is convergent, then their
limit graphon is also node-transitive.
\end{corollary}

What makes this simple assertion interesting is the fact that the
automorphism group of the limit graphon is not determined by the
automorphism groups of graphs or graphons in the convergent sequence.

\begin{example}\label{EXA:CYCLIC}
Fix any $0<\alpha<1$, and define the graph $G_n$ by $V(G_n)=[n]$,
where every $i\in[n]$ is connected to the next and previous
$\lfloor\alpha n\rfloor$ nodes (modulo $n$). The automorphism group
of $G_n$ is the dihedral group $D_n$. This sequence tends to the pure
graphon on $S^1$, with $W(x,y)=\one(\measuredangle(x,y)\le 1/2)$,
whose automorphism group is $O(2)$, the continuous version of the
dihedral groups.

No surprise so far. But let us consider the graphs $G_n\times
G_{n+1}$. Add edges connecting every node $(i,j)$ to $(i+a, j+a)$,
where $\alpha(n+1)<a<n/2$. Let $H_n$ denote the resulting graph.

Identifying node $(i,j)$ with $(i-1)n+j\in[n(n+1)]$, it is not hard
to see that $\Aut(H_n)=D_{n(n+1)}$. The limit of this graph sequence
is the pure graphon $(J,W)$, where $J$ is the torus $S^1\times S^1$,
and $W((x_1,x_2),(y_1,y_2))=\one(\measuredangle(x_1,y_1)\le 1/2,
\measuredangle(x_2,y_2)\le 1/2)$, whose automorphism group is the
wreath product of $O(2)$ with $Z_2$, a $2$-dimensional group
different from $O(2)$.
\end{example}

\begin{example}
The next example (in a slightly different form) is from the papers
\cite{Sz1} and \cite{Sz2}. It shows that even if the underlying group
$\mathbb{G}$ is the same for a convergent sequence of Cayley
graphons, a transitive action on the limit graphon may need a
different, bigger group. Let $\mathbb{G}=\mathbb{R}/\mathbb{Z}$ be
the circle group and let $\xi:\mathbb{G}\rightarrow\mathbb{C}$ be the
character defined by $\xi(x)=e^{2\pi i x}$. Let $f_n$ be the function
$\Im(1+\xi+\xi^n)/2$ where $\Im$ denotes the imaginary part. It is
not hard to see that the limit of the Cayley graphons corresponding
to $f_n$ is the Cayley graphon corresponding to the function
$f(x,y)=\Im(1+\xi(x)+\xi(y))/2$ on the torus $\mathbb{G}^2$.
\end{example}

In the light of the previous example the next theorem (which we quote
from \cite{Sz2}) is somewhat surprising. We need a definition.

\begin{definition}
Let $\mathbb{G}$ be a compact group with Haar measure $\mu$. Let
$V_n$ denote the subspace of $L^2(\mathbb{G},\mu)$ spanned by the
$\mathbb{G}$-invariant subspaces of dimension at most $n$. We say
that $\mathbb{G}$ is {\it weakly random} if $V_n$ is finite
dimensional for every $n$.
\end{definition}

\begin{theorem}\label{THM:CAYLEY}
Let $\mathbb{G}$ be a weakly random compact group. Let
$\{f_n:\mathbb{G}\rightarrow [0,1]\}_{n=1}^\infty$ be a sequence of
measurable functions such that the corresponding Cayley graphons
converge. Then the limit graphon is again a Cayley graphon on
$\mathbb{G}$.\proofend
\end{theorem}

The best known example for a weakly random group is the orthogonal
group $O(3)$. This shows that Cayley graphons on $O(3)$ behave very
differently from Cayley graphons on $O(2)$. Cayley graphons on $O(3)$
are closed with respect to graphon convergence, however Cayley
graphons on $O(2)$ are not closed.

We cite a related result, which is a consequence of a theorem of
Gowers \cite{Gow}, indicating further, more subtle, relations between
the automorphism groups of graphs and their limits.

\begin{theorem}[Gowers]\label{THM:Gowers}
Let $G_n$ be a Cayley graph of a group $\Gamma_n$ $(n=1,2,\dots)$,
where the edge-density of $G_n$ tends to a limit $0\le c\le 1$, and
the minimum dimension in which $\Gamma_n$ has a nontrivial
representation tends to infinity. Then the sequence
$(G_n)_{n=1}^\infty$ is quasirandom, i.e., it tends to a pure graphon
$(J,W)$ where $J$ has a single point.\proofend
\end{theorem}

\ignore{ The main result of this section generalizes Gowers' Theorem:
Very roughly speaking, it says that only bounded dimensional
representations of the automorphism groups matter in the limit of
node-transitive graphons. To make this precise we use the
non-standard approach to graph limit theory developed in
\cite{ElSz}.}

Our goal is to determine the automorphism group of the limit of a
sequence of node transitive graphs. Using lemma \ref{specaut} one can
reduce the problem of computing the automorphism group of $W$ to the
same problem about bounded rank graphons. To demonstrate this
principle we show the next theorem. Recall that a compact group
$\Gamma$ is {\it abelian by pro-finite} if it has a closed abelian
normal subgroup $A$ such that $\Gamma/A$ is the inverse limit of
finite groups.

\begin{theorem}
Let $\{G_n\}_{n=1}^\infty$ be a sequence of node-transitive graphs
converging to a graphon $(J,W)$. Then $(J,W)$ is weakly isomorphic to
a Cayley graphon on an abelian by pro-finite group.
\end{theorem}

\begin{proof}
We want to show that $G={\rm Aut}(W)$ has a closed, abelian by
pro-finite subgroup that acts transitively on $J$. Let
$\{\alpha_i\}_{i=1}^\infty$ be a decreasing sequence of real numbers
with $\lim_{i\to\infty}\alpha_i=0$ that contains no eigenvalue of
$W$. We can assume that $(J,W)$ is pure. Since $W$ is node
transitive, theorem \ref{THM:AUT-COMP} implies that $J$ is compact
and $\overline{J}=J$. We will use the notation from chapter
\ref{SEC:SPECPURE}.

For every $W_j$ let $\mu^j_{\alpha_i}$ denote the measure defined
above for $W_j$ in the explicit coordinate system $\mathbb{R}^{d_i}$
where $d_i=\dim(U_{\alpha_i})$. For finitely many values of $j$ the
measure $\mu^j_{\alpha_i}$ may exist in a different dimension but we
ignore those values. By choosing a subsequence we can assume without
loss of generality that the conditions of the lemma \ref{meascon}
hold for every $i$.

Let  $G^j_i\subset O(d_i)$ denote the automorphism group of
$\mu^j_{\alpha_i}$ and let $H_i$ denote the closed subgroup in
$O(d_i)$ whose elements are ultra-limits (for some fixed ultrafilter
$\omega$) of sequences $(g_1,g_2,\dots)$ where $g_j\in G^j_i$. It is
clear that elements of $H_i$ preserve $\nu_i$ and it acts
transitively on $S_i$.

We claim that $H_i$ is abelian by finite. A classical theorem by
Camille Jordan \cite{CR} states that there is a function $f(n)$ such
that any finite subgroup of $GL(n,\mathbb{C})$ contains an abelian
group of index at most $f(n)$. Using this theorem, we see that each
$G^j_i$ has an abelian subgroup of index at most $f(d_i)$. It is a
standard technique to show that this property is inherited by the
ultralimit $H_i$. If the groups $G^j_i$ are all abelian, then the
continuity of the commutator word shows that $H_i$ is abelian. For
the general case, choose $f(d_i)$ coset representatives $g_{i,j,k}$
in each group $G^j_i$ for the abelian subgroup where $1\leq k\leq
f(d_i)$. Their limits as $j\to\infty$ will be coset representatives
for the limiting abelian group.

To finish the proof, let $H$ be the inverse limit of the groups $H_i$
with respect to the homomorphisms $P_{\alpha_{i+1},\alpha_i}$. Then
$H\subseteq\Gamma_W$ and $H$ acts transitively on $S$. By lemma
\ref{specaut} we obtain that $\tau^{-1}\circ
H\circ\tau\subseteq\Aut(W)$ is transitive on $J$.
\end{proof}

\section{Graph algebras of finite rank graphons}

We conclude with an application of our results on automorphisms of
graphons to characterize graph algebras of graphons that have finite
rank as integral kernel operators. Let $(J,W)$ be a pure graphon with
finite rank. The spectral decomposition \eqref{EQ:SPEC1} takes the
simpler form
\begin{equation}\label{EQ:SPECDEC}
W(x,y)=\sum_{i=1}^t \lambda_i f_i(x)f_i(y).
\end{equation}
For any sufficiently small $\lambda>0$, we have $[W]_\lambda=W$, and
so the considerations in Section \ref{SEC:SPECPURE} imply that
$(J,r_W)$ is compact.

\ignore{
\begin{lemma}\label{LEM:FIN-COMP}
If $W$ has finite operator rank, then $(J,d_W)$ is a compact metric
space (and hence so is $(J,r_W)$).
\end{lemma}

\begin{proof}
We use the formula where the functions $f_i:J\rightarrow\mathbb{R}$
are the eigenfunctions of $W$ (pairwise orthogonal, and normalized to
have length $1$) and the numbers $\lambda_i$ are the corresponding
nonzero eigenvalues. The mapping
\[
x\mapsto v_x:=(\lambda_1f_1(x),\lambda_2f_2(x),\dots,
\lambda_tf_t(x))\in\mathbb{R}^t
\]
is an isometric embedding of $(J,d_W)$ into $\R^t$, and its range is
bounded since $\|f_i\|_\infty\leq |1/\lambda_i|$, and so
$v_x\in[-1,1]^t$ for every $x$. We know that $(J,d_W)$ is complete,
and hence it follows that it is compact.
\end{proof}
}

Let $\mathbb{G}=\Aut(W)$, and let $S$ be the function algebra
generated by the eigenfunctions of $W$. We denote by $S_n$ the space
of homogeneous polynomials of degree $n$ in the eigenfunctions of
$W$, so that $S=\bigoplus_n S_n$. Substituting \eqref{EQ:SPECDEC} in
the definition \eqref{EQ-REST-HOM} of restricted homomorphism
numbers, we see that $\AA_1\subseteq S$. Since the functions in
$\AA_1$ are $\Gbb$-invariant, it follows that $\AA_1\subseteq
S^\Gbb$. Our main goal is to prove that equality holds here.

For $h\in L^\infty(J^n)$, we define
\begin{equation}\label{EQ:HIGMULT}
r(h,x)=\int\limits_{x_1,x_2,\dots,x_n}h(x_1,x_2,\dots ,x_n)
\prod_{i=1}^n W(x,x_i).
\end{equation}
The following lemma states some elementary properties of this
function.

\begin{lemma}\label{LEM:FR1}
{\rm(a)} If $h\in L^\infty(J_n)$ then $r(h,x)\in S_n$ (as a function
of $x\in J$).

\smallskip

{\rm(b)} If $h\in\AA_n$, then $r(h,x)\in\AA_1$.

\smallskip

{\rm(c)} $r(h^g,x)=r(h,x)^g$ for every $g\in\mathbb{G}$.
\end{lemma}

\begin{proof}
Assertion (a) follows by substituting formula (\ref{EQ:SPECDEC}) in
(\ref{EQ:HIGMULT}). To prove (b), let $h(x_1,\dots,x_n)=t_{x_1\dots
x_n}(s,W)$, and let $s'\in\QQ_1$ denote the one-labeled quantum graph
obtained from $s$ by connecting a new node with label $1$ to all the
labeled nodes and then we removing the original labels. Then
$r(h,x)=t_x(s',W)$. Finally, (c) follows by replacing $W(x,x_i)$ by
$W(x^g,x_i^g)$ in the formula for $r(h^g,x)$. Since the action of
$\mathbb{G}$ is measure preserving, the integration over
$(x_1^g,x_2^g,\dots ,x_n^g)$ is equivalent to the integration over
$(x_1,x_2,\dots,x_n)$.
\end{proof}

\begin{lemma}\label{LEM:FR2}
Every function $f\in S_n$ can be expressed as $f=r(h,x)$ for some
function $h\in L^\infty(J^n)$.
\end{lemma}

\begin{proof}
If $h(x_1,x_2,\dots ,x_n)=f_{i_1}(x_1)f_{i_2}(x_2)\dots
f_{i_n}(x_n)$, then
\[
r(h,x)=\lambda_{i_1}\lambda_{i_2}\dots \lambda_{i_n}
f_{i_1}(x)f_{i_2}(x)\dots f_{i_n}(x).
\]
Every function $f\in S_n$ can be expressed as a linear combination of
functions such as that on the right side of the previous formula.
Since $r(h,x)$ is linear in $h$, this completes the proof.
\end{proof}

\begin{lemma}\label{LEM:FR3}
Every function $f\in S_n^\mathbb{G}$ can be expressed as $f=r(h,x)$
for some $\mathbb{G}$-invariant function $h\in L^\infty(J^n)$.
\end{lemma}

\begin{proof}
By Lemma \ref{LEM:FR2}, $f(x)=r(q,x)$ for a suitable $q\in
L^\infty(J^n)$. Let $h=\int_\mathbb{G} q^g$. It is clear that $h$ is
$\mathbb{G}$-invariant. By lemma \ref{LEM:FR1} and the linearity of
$r(.\,,.)$ in the first variable it follows that $f(x)=r(h,x)$.
\end{proof}

\begin{lemma}
$S_n^\mathbb{G}=\AA_1\cap S_n$.
\end{lemma}

\begin{proof}
Trivially $S_n^\mathbb{G}\supseteq \AA_1\cap S_n$. To prove the
reverse, let $f\in S_n^\mathbb{G}$. Trivially $f\in S_n$, so it
suffices to prove that $f\in\AA_1$. Lemma \ref{LEM:FR3} shows that
$f(x)=r(h,x)$ for some $\mathbb{G}$-invariant function $h\in
L^\infty(J^n)$. Using Corollary \ref{COR:ALGINV}, there is a sequence
of functions $q_k\in\AA^0_n$ such that $q_k\to h$ ($k\to\infty$)
uniformly in $x$. By Lemma \ref{LEM:FR1}, $r(q_k,x)\in\AA_1$ (as a
function of $x\in J$), and clearly $r(q_k,x)\to f=r(h,x)$
uniformly in $x$. This implies that $f\in\AA_1$.
\end{proof}

\begin{theorem}
$\AA_1=S^\mathbb{G}$.
\end{theorem}

\begin{proof}
We have seen that $\AA_1\subseteq S^\mathbb{G}$. To prove the
reverse, we note that every function $f\in S$ is a finite sum of
functions $\sum_n f_n$, where $f_n\in S_n$, and if $f$ is
$\Gbb$-invariant, then so are the terms $f_n$. Hence $S^\mathbb{G}$
is the linear span of the spaces $S_n^\mathbb{G}$. By the previous
lemma we get that $S^\mathbb{G}\subseteq\AA_1$.
\end{proof}

\begin{corollary}
$\AA_1$ is finitely generated.
\end{corollary}

\begin{proof}
The algebra $S$ is a finitely generated commutative algebra and the
compact group $\mathbb{G}$ acts on $S$ via automorphisms. Hilbert's
theorem on $\Gbb$-invariant rings implies that
$\mathbb{A}_1=S^\mathbb{G}$ is finitely generated.
\end{proof}

\end{document}